\documentclass[a4paper,10pt]{article}

\usepackage[english]{babel}
\usepackage{amsmath,amsfonts,amssymb,amsthm}
\usepackage{mathrsfs}
\usepackage{bbm}
\usepackage{indentfirst}

\usepackage{mathtools}
\mathtoolsset{showonlyrefs}

\usepackage{graphicx}
\usepackage{pgf,tikz}
\usepackage{enumitem}

\newtheorem{theorem}{Theorem}[section]

\newtheorem{corollary}{Corollary}[section]

\theoremstyle{definition}

\numberwithin{equation}{section}

\newcommand\blfootnote[1]{\begingroup\renewcommand\thefootnote{}\footnote{#1}\addtocounter{footnote}{-1}\endgroup}

\begin{document}

\title{
{\bf\Large Multiple positive solutions \break of a Sturm-Liouville boundary value problem \break with conflicting nonlinearities}
\footnote{Work supported by the Grup\-po Na\-zio\-na\-le per l'Anali\-si Ma\-te\-ma\-ti\-ca, la Pro\-ba\-bi\-li\-t\`{a} e le lo\-ro
Appli\-ca\-zio\-ni (GNAMPA) of the Isti\-tu\-to Na\-zio\-na\-le di Al\-ta Ma\-te\-ma\-ti\-ca (INdAM).
Progetto di Ricerca 2016: ``Problemi differenziali non lineari: esistenza, molteplicit\`{a} e propriet\`{a} qualitative delle soluzioni''.}
}
\author{{\bf\large Guglielmo Feltrin}
\vspace{1mm}\\
{\it\small SISSA - International School for Advanced Studies}\\
{\it\small via Bonomea 265}, {\it\small 34136 Trieste, Italy}\\
{\it\small e-mail: guglielmo.feltrin@sissa.it}\vspace{1mm}
}

\date{}

\maketitle

\vspace{-2mm}

\begin{abstract}
\noindent
We study the second order nonlinear differential equation
\begin{equation*}
u''+ \sum_{i=1}^{m} \alpha_{i} a_{i}(x)g_{i}(u) - \sum_{j=0}^{m+1} \beta_{j} b_{j}(x)k_{j}(u) = 0,
\end{equation*}
where $\alpha_{i},\beta_{j}>0$, $a_{i}(x), b_{j}(x)$ are non-negative Lebesgue integrable functions defined in $\mathopen{[}0,L\mathclose{]}$, and
the nonlinearities $g_{i}(s), k_{j}(s)$ are continuous, positive and satisfy suitable growth conditions, as to cover the classical superlinear equation $u''+a(x)u^{p}=0$, with $p>1$.
When the positive parameters $\beta_{j}$ are sufficiently large, we prove the existence of at least $2^{m}-1$
positive solutions for the Sturm-Liouville boundary value problems associated with the equation.
The proof is based on the Leray-Schauder topological degree for locally compact operators on open and possibly unbounded sets.
Finally, we deal with radially symmetric positive solutions for the Dirichlet problems associated with elliptic PDEs.
\end{abstract}

\section{Introduction}\label{section-1}

In this paper we study positive solutions to nonlinear second order ODEs 
with indefinite weight and we deal with Sturm-Liouville boundary conditions.
To describe our results, throughout the introduction we focus our attention to the equation
\begin{equation}\label{eq-intro}
u''+ a(x)g(u) - \mu \, b(x)k(u) = 0
\end{equation}
defined on the nontrivial compact interval $\mathopen{[}0,L\mathclose{]}$.
Let $\mathbb{R}^{+}:=\mathopen{[}0,+\infty\mathclose{[}$ be the set of non-negative real numbers.
We assume that $\mu>0$ is a real parameter, $a,b\colon \mathopen{[}0,L\mathclose{]} \to \mathbb{R}^{+}$ are measurable functions 
and $g,k\colon \mathbb{R}^{+} \to \mathbb{R}^{+}$ are continuous functions such that
\begin{equation*}
\begin{aligned}
& g(0)=0, &\quad g(s)>0, \quad \text{for } s>0, \\
& k(0)=0, &\quad k(s)>0, \quad \text{for } s>0.
\end{aligned}
\leqno{(i_{1})}
\end{equation*}
Referring to \cite{Ru-98}, we can say that equation \eqref{eq-intro} exhibits \textit{conflicting nonlinearities}.
Moreover, following a standard terminology (cf.~\cite{BoFeZa-17tams}), we can look at \eqref{eq-intro} as 
an \textit{indefinite} equation, meaning that the sign of the weight is non-constant.

Our main goal is to provide multiplicity results of positive solutions to equation \eqref{eq-intro} 
together with the Sturm-Liouville boundary conditions, namely conditions of the form
\begin{equation}\label{BC-intro}
\begin{cases}
\, \alpha u(0) - \beta u'(0) = 0 \\
\, \gamma u(L) + \delta u'(L) = 0,
\end{cases}
\end{equation}
where $\alpha, \beta, \gamma, \delta \geq 0$ with $\gamma\beta + \alpha\gamma +\alpha\delta > 0$.
We notice that for $\alpha=\gamma=1$ and $\beta=\delta=0$, we obtain the Dirichlet boundary conditions.

A \textit{solution} of \eqref{eq-intro} is an absolutely continuous function $u\colon\mathopen{[}0,L\mathclose{]}\to {\mathbb{R}}^{+}$
such that its derivative $u'(x)$ is absolutely continuous and $u(x)$ satisfies \eqref{eq-intro} for a.e. $x\in\mathopen{[}0,L\mathclose{]}$.
We look for \textit{positive solutions} of boundary value problem \eqref{eq-intro}-\eqref{BC-intro}, that is solutions $u(x)$ of \eqref{eq-intro} satisfying \eqref{BC-intro} and 
such that $u(x)>0$ for every $x\in\mathopen{]}0,L\mathclose{[}$.

\medskip

Starting from the Seventies, these types of problems have received a remarkable attention in the
research area of nonlinear differential equations.
One of the early work was due to Anderson (cf.~\cite{An-71}) who has proved that the equation 
\begin{equation*}
-\Delta \,u = u^{3} - \mu u^{5} -u  \quad \text{in } \mathbb{R}^{N}
\end{equation*}
has a solution if $0<\mu<3/16$, while there are no solutions for $\mu>3/16$.

Other two relevant contributions to the autonomous case are \cite{AmBrCe-94,BeLi-83}.
In \cite{BeLi-83} Berestycki and Lions have analyzed the more general equation
\begin{equation*}
-\Delta \,u = \nu |u|^{p-1}u - \mu |u|^{q-1}u -\lambda u \quad \text{in } \mathbb{R}^{N},
\end{equation*}
where $N\geq 3$, $\nu,\mu,\lambda>0$ and $1<q<p<(N+2)/(N-2)$, and they
proved existence and non-existence results in dependence of the parameter $\mu>0$.
In \cite{AmBrCe-94} Ambrosetti, Brezis and Cerami proved that there is a positive solution of
\begin{equation*}
\begin{cases}
\, -\Delta \,u = \lambda u^{q} + u^{p} & \text{in } \Omega \\
\, u = 0 & \text{on } \partial\Omega,
\end{cases}
\end{equation*}
with $0<q<1<p$, for $\lambda>0$ small enough and no solution for $\lambda$ large.

We refer to \cite{Ru-98} for a further result in this direction and for a more complete presentation and bibliography on the subject.

\medskip

Our research work has been motivated by the papers \cite{AlTa-96,GiGo-09prse}, where non-autonomous
differential equations on bounded domains are taken into account.
The boundedness of the domain enables the authors to deal with more general equations (whit respect to those considered in \cite{AmBrCe-94,An-71,BeLi-83}) and, in particular,
to consider non-negative weight functions in place of the positive coefficients in front of the nonlinearities.

In \cite{AlTa-96} Alama and Tarantello studied positive solutions of the Dirichlet boundary value problem
\begin{equation*}
\begin{cases}
\, -\Delta \,u = \lambda u +k(x)u^{q} - h(x) u^{p} & \text{in } \Omega \\
\, u = 0 & \text{on } \partial\Omega,
\end{cases}
\end{equation*}
where $\Omega\subseteq\mathbb{R}^{N}$ (with $N\geq 3$) is an open bounded set with smooth boundary,
the functions $h,k\in L^{1}(\Omega)$ are non-negative and $1<q<p$.
They proved existence, non-existence and multiplicity results depending on $\lambda\in\mathbb{R}$
and according to the properties of the ratio $k^{p-1}/h^{q-1}$.

In \cite{GiGo-09prse} Gir\~{a}o and Gomes dealt with nodal solutions to
\begin{equation*}
\begin{cases}
\, -\Delta \,u = a^{+}(x) \bigl{(}\lambda u + f(x,u)\bigr{)} - \mu a^{+}(x) g(x,u) & \text{in } \Omega \\
\, u = 0 & \text{on } \partial\Omega,
\end{cases}
\end{equation*}
where $\Omega\subseteq\mathbb{R}^{N}$ (with $N\geq 1$) is an open bounded set with smooth boundary.
They proved existence of nodal solutions for $\mu>0$ sufficiently large.

\medskip

The main goal of this paper is to present a multiplicity result for positive solutions to \eqref{eq-intro}-\eqref{BC-intro}
in dependence of the number of the intervals where $a(x)>0$ and thus giving a contribution to \cite{AlTa-96,GiGo-09prse}.
In order to explain our achievement, we now introduce it in a slightly easier framework.

Let $a,b\colon \mathopen{[}0,L\mathclose{]} \to \mathbb{R}^{+}$ be continuous functions such that
\begin{itemize}
\item [$(i_{2})$] there exist two zeros $\tau,\sigma$ with $0<\tau<\sigma<L$ such that
\begin{equation*}
\begin{aligned}
&a(x)>0 &\text{on }& \mathopen{]}0,\tau\mathclose{[}\cup\mathopen{]}\sigma,L\mathclose{[}, \quad & &a(x)\equiv 0  &\text{on }& \mathopen{[}\tau,\sigma\mathclose{]},\\
&b(x)>0 &\text{on }& \mathopen{]}\tau,\sigma\mathclose{[},& &b(x)\equiv 0 &\text{on }& \mathopen{[}0,\tau\mathclose{]}\cup\mathopen{[}\sigma,L\mathclose{]}.
\end{aligned}
\end{equation*}
\end{itemize}

Our main multiplicity result is the following. See Figure~\ref{fig-01} for a numerical example.

\begin{theorem}\label{th-intro}
Let $a,b\colon \mathopen{[}0,L\mathclose{]} \to \mathbb{R}^{+}$ be continuous functions satisfying $(i_{2})$.
Let $g,k\colon \mathbb{R}^{+} \to \mathbb{R}^{+}$ be continuous functions satisfying $(i_{1})$.
Moreover, assume that
\begin{equation*}
\limsup_{s\to0^{+}}\dfrac{g(s)}{s}=0, \qquad \liminf_{s\to +\infty}\dfrac{g(s)}{s}=+\infty,
\end{equation*}
and
\begin{equation*}
\limsup_{s\to0^{+}}\dfrac{k(s)}{s}<+\infty.
\end{equation*}
Then there exists $\mu^{*}>0$ such that for every $\mu >\mu^{*}$ the boundary value problem \eqref{eq-intro}-\eqref{BC-intro}
has at least $3$ positive solutions.
\end{theorem}

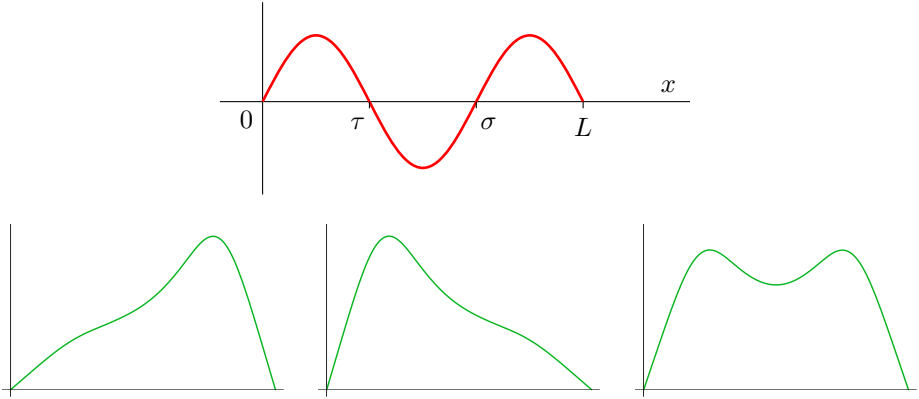
\begin{figure}[h!]
\centering
\begin{tikzpicture}[x=40pt,y=25pt]
\draw (-0.4,0) -- (4,0);
\draw (0,-1.4) -- (0,1.5);
\draw (3.8,0) node[anchor=south] {$x$};
\draw [line width=1pt, color=red] (0,0) sin (0.5,1) cos (1,0) sin (1.5,-1) cos (2,0) sin (2.5,1) cos (3,0);
\draw (0,0) node[anchor=north east] {$0$};
\draw (1.05,-0.1) node[anchor=north east] {$\tau$};
\draw (1.95,-0.1) node[anchor=north west] {$\sigma$};
\draw (3,-0.1) node[anchor=north] {$L$};
\draw (1,-0.1) -- (1,0);
\draw (2,-0.1) -- (2,0);
\draw (3,-0.1) -- (3,0);
\end{tikzpicture}
\vspace*{10pt}
\\
\includegraphics[width=0.305\textwidth]{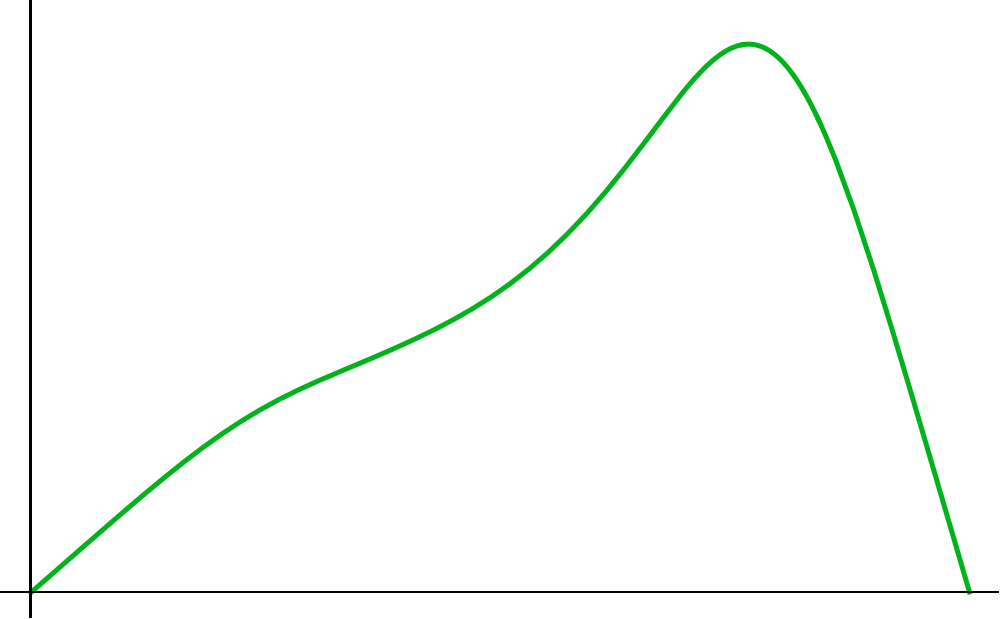} \quad
\includegraphics[width=0.305\textwidth]{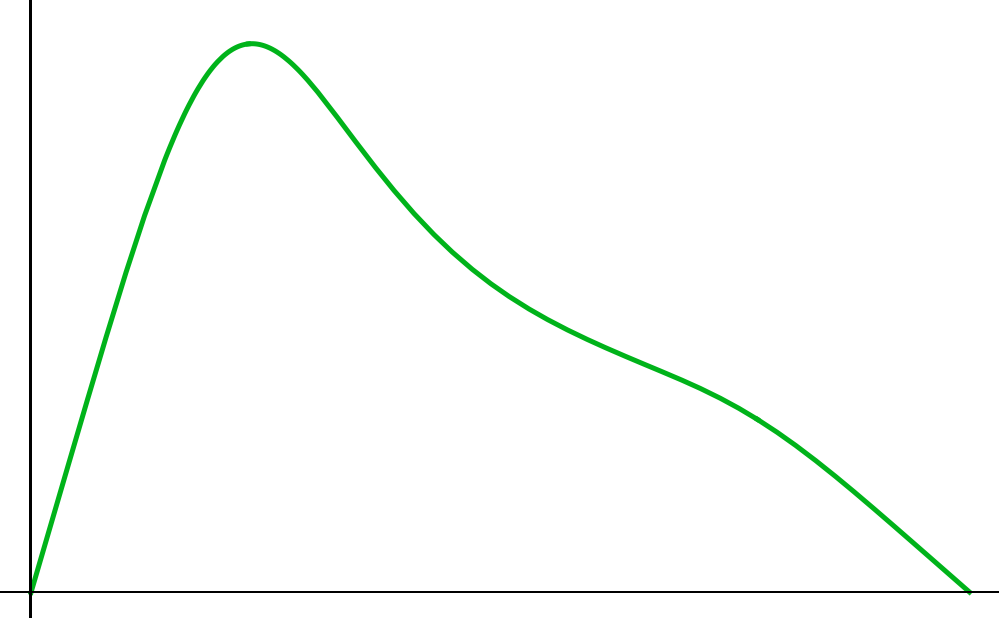} \quad
\includegraphics[width=0.305\textwidth]{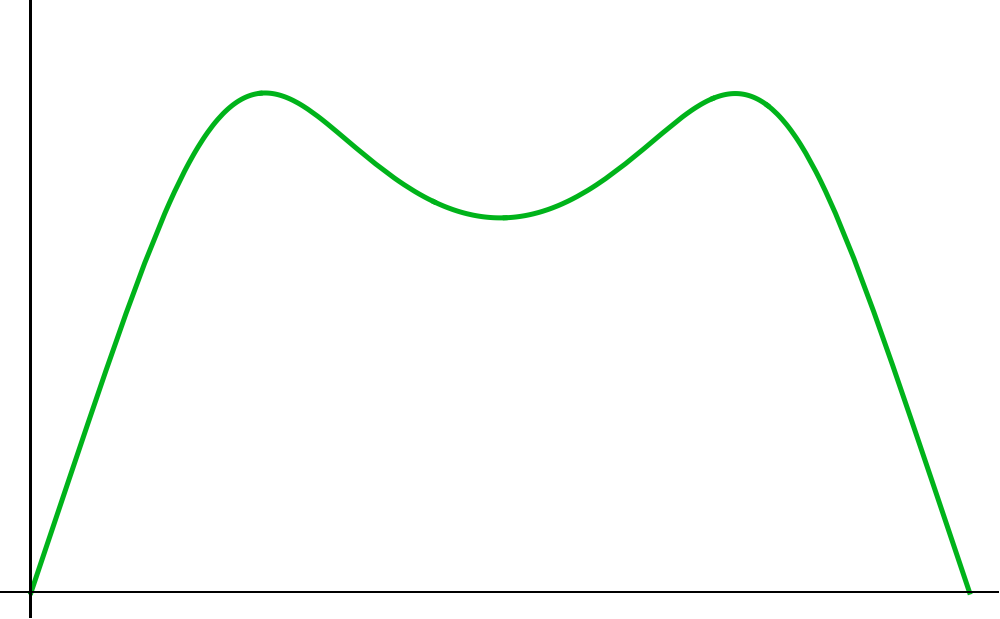}
\caption{\small{The figure shows an example of $3$ positive solutions to the Dirichlet problem associated with \protect\eqref{eq-intro} on $\mathopen{[}0,3\pi\mathclose{]}$,
where $\tau=\pi$, $\sigma=2\pi$, $L=3\pi$, $a(x)=\sin^{+}(x)$, $b(x)=\sin^{-}(x)$ (as in the upper part of the figure), $g(s)=s^{2}$, $k(s)=s^{3}$ (for $s>0$).
For $\mu=1$, Theorem~\ref{th-intro} ensures the existence of $3$ positive solutions, whose graphs are located in the lower part of the figure.
}}
\label{fig-01}
\end{figure}

We notice that, under the hypotheses of Theorem~\ref{th-intro}, the function $g(s)$ is \textit{superlinear}, thus covering the classical case $g(s)=s^{p}$ with $p>1$.
On the other hand, we do not impose any growth condition on $k(s)$. Hence, the case considered in \cite{AlTa-96} is clearly included in our setting
(cf.~Section~\ref{section-6} for other remarks in this direction).

The proof of Theorem~\ref{th-intro} and its variants is based on the Leray-Schauder topological degree.
Our work benefits from a topological approach which has been first employed in \cite{FeZa-15jde} by Feltrin and Zanolin
to produce multiple positive solutions to the Dirichlet problem
\begin{equation}\label{BVP-jde}
\begin{cases}
\, u'' + a(x)g(u) = 0 \\
\, u(0) = u(L) = 0.
\end{cases}
\end{equation}
The main idea is to transform our Sturm-Liouville boundary value problem \eqref{eq-intro}-\eqref{BC-intro}
into an equivalent fixed point problem in the Banach space $X:=\mathcal{C}(\mathopen{[}0,L\mathclose{]})$
\begin{equation*}
u = \Phi u, \quad u\in X,
\end{equation*}
where $\Phi\colon X\to X$ is a completely continuous operator. 
As a crucial step in the proof, we define $3$ open pairwise disjoint subsets of $X$ not containing the trivial solution
and we prove that the degree is well defined and different from zero on these sets.
This fact guarantees, via a maximum principle, that there exist at least $3$ positive solutions of \eqref{eq-intro}-\eqref{BC-intro}.

\medskip

As remarked above, Theorem~\ref{th-intro} is a special case of the main result of the present paper (cf.~Theorem~\ref{main-theorem}),
where we deal with more general (Lebesgue integrable) coefficients $a(x)$ and $b(x)$ and weaker growth conditions on $g_{i}(s)$. Roughly speaking, we consider a weight function $a\colon \mathopen{[}0,L\mathclose{]} \to \mathbb{R}^{+}$
(belonging to the ``positive'' part of the nonlinearity) which is positive on $m$ intervals, so $a(x)$ has $m$ positive humps.
In this framework we prove the existence of $2^{m}-1$ positive solutions of \eqref{eq-intro}-\eqref{BC-intro} when $b(x)$ is ``sufficiently large'', namely $\beta\gg0$.

Our result is a new contribution on the line of research initiated by G\'{o}mez-Re\~{n}asco and L\'{o}pez-G\'{o}mez in \cite{GRLG-00},
where the authors (supported by numerical evidence) conjectured the fact that there exist at least $2^{m}-1$ positive solutions for the Dirichlet problem \eqref{BVP-jde}
when $a(x)$ has $m$ positive humps separated by $m-1$ negative ones and the negative part $a^{-}(x)$ is large enough. 
For the Dirichlet problem, contributions in this direction have been achieved in \cite{BoGoHa-05, FeZa-15jde, GaHaZa-03mod, GiGo-09prse} (see also the references therein).
At the best of our knowledge, this is the first work addressing the same questions for a problem with conflicting nonlinearities and Sturm-Liouville boundary conditions.

\medskip 

The plan of the paper is as follows. In Section~\ref{section-2} we introduce all the hypotheses on the elements involved in 
\eqref{eq-intro}-\eqref{BC-intro} which are assumed for the rest of the paper, moreover we define the subsets of $X$
where we compute the Leray-Schauder degree.
In Section~\ref{section-3} we briefly recall and adapt the main result in \cite{FeZa-15jde} in the context of
Sturm-Liouville boundary value problems.
Section~\ref{section-4} is devoted to our main result (cf.~Theorem~\ref{main-theorem}). Taking advantage of the discussion in Section~\ref{section-3},
we give the proof and we present some consequences.
In Section~\ref{section-5} we provide an application to radially symmetric solutions of elliptic PDEs on annular domains.
The final Section~\ref{section-6} is dedicated to some conclusive comments and possible variants of Theorem~\ref{main-theorem}.

\section{Setting and notation}\label{section-2}

In this section we present the main elements involved in the study of the positive solutions to the boundary value problem
\begin{equation}\label{BVP}
\begin{cases}
\, u'' + f(x,u) = 0 \\
\, \alpha u(0) -\beta u'(0) = 0 \\
\, \gamma u(L) + \delta u'(L) = 0,
\end{cases}
\end{equation}
where $f\colon\mathopen{[}0,T\mathclose{]}\times \mathbb{R}^{+}\to\mathbb{R}$ is a function of the form
\begin{equation}\label{def_f}
f(x,s):= \sum_{i=1}^{m} \alpha_{i} a_{i}(x)g_{i}(s) - \sum_{j=0}^{m+1} \beta_{j} b_{j}(x)k_{j}(s),
\end{equation}
with $m\geq1$, and $\alpha, \beta, \gamma, \delta \geq 0$ with $\gamma\beta + \alpha\gamma +\alpha\delta > 0$.

The following hypotheses and positions will be assumed from now on in the paper.

Let $m\geq1$ be an integer.
Let $\alpha_{i}>0$, for $i=1,\ldots,m$, 
and $\beta_{j}>0$, for $j=0,\ldots,m+1$, be real parameters.

Let $a_{i}\colon \mathopen{[}0,L\mathclose{]} \to \mathbb{R}^{+}$, for $i=1,\ldots,m$, 
and $b_{j} \colon \mathopen{[}0,L\mathclose{]} \to \mathbb{R}^{+}$, for $j=0,\ldots,m+1$, be (non-negative) Lebesgue integrable functions.
Moreover, we assume that
\begin{itemize}
\item [$(h_{1})$] 
\textit{there exist $2m+2$ closed and pairwise disjoint intervals $I_{1},\ldots,I_{m}$ and
$J_{0},\ldots,J_{m+1}$ ($J_{0}$ and $J_{m+1}$ possibly empty), such that
\begin{equation*}
\begin{aligned}
&a_{i}\not\equiv 0 \; \text{ on } I_{i}, & &a_{i}\equiv 0  \;  \text{ on } \mathopen{[}0,L\mathclose{]}\setminus I_{i},&  &i=1,\ldots,m; \\
&b_{j}\not\equiv 0 \;  \text{ on } J_{j},& &b_{j}\equiv 0  \;  \text{ on } \mathopen{[}0,L\mathclose{]}\setminus J_{j},&  &j=0,1,\ldots,m+1.
\end{aligned}
\end{equation*}
}
\end{itemize}
Without loss of generality, up to a relabelling of the indices, we can assume that $\max I_{i} \leq \min I_{k}$, for all $i<k$; $\max J_{j} \leq \min J_{k}$, for all $j<k$;
$\max I_{i} \leq \min J_{j}$, for all $i<j$; between two intervals $I_{i}$ and $I_{i+1}$ there is an interval $J_{j}$;
between two intervals $J_{j}$ and $J_{j+1}$ there is an interval $I_{i}$. Moreover, eventually extending the functions $a_{i}(x)$ as $0$
on $(\mathopen{[}0,L\mathclose{]}\setminus I_{i})\cap \mathopen{[}\max J_{j},\min J_{j+1}\mathclose{]}$ (with $I_{i}$ between $J_{j}$ and $J_{j+1}$), we can also suppose that
\begin{equation*}
\bigcup_{i=1}^{m} I_{i} \cup \bigcup_{j=0}^{m+1} J_{j} = \mathopen{[}0,L\mathclose{]}.
\end{equation*}
Summarizing all the conventions, it is not restrictive to label the intervals $I_{i}$ and $J_{j}$ following the natural order given by the
standard orientation of the real line and thus determine $2m + 2$ points
\begin{equation*}
0 = \tau_{0} \leq \sigma_{1} < \tau_{1} < \sigma_{2} < \tau_{2} < \ldots < \sigma_{m-1} < \tau_{m-1} < \sigma_{m} < \tau_{m} \leq \sigma_{m+1} = L,
\end{equation*}
so that
\begin{equation*}
I_{i} := \mathopen{[}\sigma_{i},\tau_{i}\mathclose{]}, \quad i=1,\ldots,m, \quad \text{ and } \quad
J_{j} := \mathopen{[}\tau_{j},\sigma_{j+1}\mathclose{]}, \quad j=0,\ldots,m+1.
\end{equation*}
Finally, consistently with assumption $(h_{1})$ and without loss of generality, we select the points $\sigma_{i}$
and $\tau_{i}$ in such a manner that $b_{j}(x)\not\equiv0$  on all right neighborhoods of $\tau_{j}$ and on all left neighborhoods of $\sigma_{j+1}$.
In other words, if there is an interval $K$ contained in $\mathopen{[}0,L\mathclose{]}$ where $a(x)\equiv 0$,
we choose the points $\sigma_{i}$ and $\tau_{i}$ so that $K$ is contained in one of the $I_{i}$
or $K$ is contained in the interior of one of the $J_{j}$.

\medskip

Let $g_{i}\colon \mathbb{R}^{+} \to \mathbb{R}^{+}$, for $i=1,\ldots,m$, 
and $k_{j}\colon \mathbb{R}^{+} \to \mathbb{R}^{+}$, $j=0,\ldots,m+1$, be continuous functions and such that
\begin{equation*}
\begin{aligned}
& g_{i}(0)=0, &\quad g_{i}(s)>0, \quad \text{for } s>0, &\qquad i=1,\ldots,m; \\
& k_{j}(0)=0, &\quad k_{j}(s)>0, \quad \text{for } s>0, &\qquad j=0,\ldots,m+1.
\end{aligned}
\leqno{(h_{2})}
\end{equation*}
We define
\begin{equation*}
g_{0}^{i}:= \limsup_{s\to0^{+}}\dfrac{g_{i}(s)}{s}, \quad g_{\infty}^{i} := \liminf_{s\to +\infty}\dfrac{g_{i}(s)}{s}, \qquad i=1,\ldots,m,
\end{equation*}
and 
\begin{equation*}
k_{0}^{j}:= \limsup_{s\to0^{+}}\dfrac{k_{j}(s)}{s}, \qquad j=0,\ldots,m+1.
\end{equation*}
For all $i=1,\ldots,m$ and for all $j=0,\ldots,m+1$, we suppose
\begin{equation*}
g_{0}^{i} < +\infty, \qquad g_{\infty}^{i}>0, \qquad k_{0}^{j}< +\infty.
\leqno{(h_{3})}
\end{equation*}
We denote with $\lambda_{0}$ the first (positive) eigenvalue of the eigenvalue problem
\begin{equation*}
\begin{cases}
\, \varphi'' +\lambda \bigl{[}\sum_{i=1}^{m}a_{i}(x)\bigr{]} \varphi = 0 \\
\, \alpha \varphi(0) -\beta \varphi'(0) = 0 \\
\, \gamma \varphi(L) + \delta \varphi'(L) = 0,
\end{cases}
\end{equation*}
and, for $i=1,\ldots,m$, with $\lambda_{1}^{i}$ the first eigenvalue of the eigenvalue problem in $I_{i}$
\begin{equation*}
\begin{cases}
\, \varphi'' + \lambda a_{i}(x) \varphi = 0 \\
\, \varphi|_{\partial I_{i}}=0.
\end{cases}
\end{equation*}
If $\tau_{0}=\sigma_{1}=0$ or $\tau_{m}=\sigma_{m+1}=L$, we denote with $\lambda_{1}^{i}$ (with $i=1$ or $i=m$, respectively) the first eigenvalue of the eigenvalue problem
\begin{equation*}
\begin{cases}
\, \varphi'' + \lambda a_{i}(x) \varphi = 0 \\
\, \alpha \varphi(0) -\beta \varphi'(0) = 0 \\
\, \varphi(\tau_{1}) = 0
\end{cases}
\quad \text{ or } \quad
\begin{cases}
\, \varphi'' + \lambda a_{i}(x) \varphi = 0 \\
\, \varphi(\sigma_{m})= 0 \\
\, \gamma \varphi(L) + \delta \varphi'(L) = 0,
\end{cases}
\end{equation*}
respectively. Clearly, if $\beta=0$ or $\delta=0$, respectively, the definition of $\lambda_{1}^{i}$ is the same as before.
Using the assumptions on $a_{i}(x)$, in any case, we obtain that $\lambda_{1}^{i} > 0$ for each $i=1,\ldots,m$.

\medskip

Finally, we introduce some open subsets of the Banach space $\mathcal{C}(\mathopen{[}0,T\mathclose{]})$, endowed with the $\sup$-norm $\|\cdot\|_{\infty}$.
Let $\mathcal{I}\subseteq\{1,\ldots,m\}$ be a subset of indices (possibly empty) and let
$d,D$ be two fixed positive real numbers with $d<D$.
We define two families of open unbounded sets
\begin{equation}\label{eq-Omega}
\begin{aligned}
\Omega^{\mathcal{I}}_{d,D}:=
\biggl{\{}\,u\in\mathcal{C}(\mathopen{[}0,T\mathclose{]})\colon
   & \max_{x\in I_{i}}|u(x)|<D, \, i\in\mathcal{I};                             &
\\ & \max_{x\in I_{i}}|u(x)|<d, \, i\in\{1,\ldots,m\}\setminus\mathcal{I}       & \biggr{\}}
\end{aligned}
\end{equation}
and
\begin{equation}\label{eq-Lambda}
\begin{aligned}
\Lambda^{\mathcal{I}}_{d,D}:=
\biggl{\{}\,u\in\mathcal{C}(\mathopen{[}0,T\mathclose{]})\colon
   & d < \max_{x\in I_{i}}|u(x)|<D, \, i\in\mathcal{I};                           &
\\ & \max_{x\in I_{i}}|u(x)|<d, \, i\in\{1,\ldots,m\}\setminus\mathcal{I}     & \biggr{\}}.
\end{aligned}
\end{equation}
In the sequel, once the constants $d$ and $D$ are fixed, we simply use the symbols $\Omega^{\mathcal{I}}$ and $\Lambda^{\mathcal{I}}$
to denote $\Omega^{\mathcal{I}}_{d,D}$ and $\Lambda^{\mathcal{I}}_{d,D}$, respectively.

\section{Reviewing a previous multiplicity result}\label{section-3}

In this section we briefly recall the main result obtained by Feltrin and Zanolin in \cite{FeZa-15jde} concerning multiplicity of positive solutions to \eqref{BVP}.
More precisely, in \cite{FeZa-15jde} the authors dealt with positive solutions of the Dirichlet boundary value problem (i.e.~problem \eqref{BVP} with $\alpha=\gamma=1$ and $\beta=\delta=0$).
Subsequently, in \cite[\S~5.4]{FeZa-15jde} they observed that the approach presented therein could be adapted to the study of different boundary conditions, for example 
$u(0)=u'(L)=0$ or $u'(0)=u(L)=0$, which are clearly covered by the ones discussed in the present paper.
This section is devoted to the presentation of the multiplicity result in \cite{FeZa-15jde} in the context of a Sturm-Liouville boundary value problem.
Since there are some difference in considering the Sturm-Liouville boundary conditions, we will give a sketch of the proof.

\medskip

Let us consider a general map $f\colon \mathopen{[}0,L\mathclose{]}\times {\mathbb{R}}^{+}\to {\mathbb{R}}$ and suppose that $f(x,s)$ is an $L^{1}$-Carath\'{e}odory function, that is
$x\mapsto f(x,s)$ is measurable for each $s\in{\mathbb{R}}^{+}$, $s\mapsto f(x,s)$ is continuous for a.e.~$x\in\mathopen{[}0,L\mathclose{]}$,
for each $d>0$ there is $\eta_{d}\in L^{1}(\mathopen{[}0,L\mathclose{]},{\mathbb{R}}^{+})$ such that $|f(x,s)|\leq\eta_{d}(x)$, for a.e.~$x\in\mathopen{[}0,L\mathclose{]}$ and for all $|s|\leq d$.

In order to state the multiplicity result we list the following hypotheses that will be assumed.
\begin{itemize}
\item [$(f^{*})$] 
$f(x,0) = 0$, for a.e.~$x\in \mathopen{[}0,L\mathclose{]}$.
\item [$(f_{0}^{-})$]
There exists a function $q_{-}\in L^{1}(\mathopen{[}0,L\mathclose{]},{\mathbb{R}}^{+})$ such that
\begin{equation*}
\liminf_{s\to0^{+}}\dfrac{f(x,s)}{s}\geq -q_{-}(x), \quad \text{uniformly a.e. } x\in\mathopen{[}0,L\mathclose{]}.
\end{equation*}
\item [$(f_{0}^{+})$]
There exists a function
$q_{0}\in L^{1}(\mathopen{[}0,L\mathclose{]},{\mathbb{R}}^{+})$ with $q_{0}\not\equiv0$ such that
\begin{equation*}
\limsup_{s\to 0^{+}}\dfrac{f(x,s)}{s}\leq q_{0}(x), \quad \text{uniformly a.e. } x\in \mathopen{[}0,L\mathclose{]},
\end{equation*}
and
\begin{equation*}
\mu_{1}(q_{0})>1,
\end{equation*}
where $\mu_{1}(q_{0})$ is the first positive eigenvalue of the eigenvalue problem
\begin{equation*}
\varphi''+\mu q_{0}(x)\varphi=0, \qquad \varphi(0)=\varphi(L)=0.
\end{equation*}
\item [$(H)$]
There exist $m\geq 1$ intervals $I_{1},\ldots,I_{m}$, closed and pairwise disjoint, such that
\begin{equation*}
\begin{aligned}
   & f(x,s) \geq 0, \quad \text{for a.e. }  x\in \bigcup_{i=1}^{m} I_{i} \mbox{ and for all } s\geq 0;
\\ & f(x,s) \leq 0, \quad \text{for a.e. }  x\in \mathopen{[}0,L\mathclose{]}\setminus \bigcup_{i=1}^{m} I_{i} \mbox{ and for all } s\geq 0.
\end{aligned}
\end{equation*}
\item [$(f_{\infty})$]
For all $i=1,\ldots,m$ there exists a function $q_{\infty}^{i}\in L^{1}(I_{i},{\mathbb{R}}^{+})$
with $q_{\infty}^{i}\not\equiv0$ such that
\begin{equation*}
\liminf_{s\to+\infty}\dfrac{f(x,s)}{s}\geq q_{\infty}^{i}(x), \quad \text{uniformly a.e. } x\in I_{i},
\end{equation*}
and
\begin{equation*}
\mu_{1}^{I_{i}}(q_{\infty}^{i})<1,
\end{equation*}
where $\mu_{1}^{I_{i}}(q_{\infty}^{i})$ is the first positive eigenvalue of the eigenvalue problem in $I_{i}$
\begin{equation*}
\varphi''+\mu q_{\infty}^{i}(x)\varphi=0, \qquad \varphi|_{\partial I_{i}}=0.
\end{equation*}
\end{itemize}

We observe that, since $f(x,s)$ satisfies condition $(f_{0}^{+})$, from the continuity
of the eigenvalue $\mu_{1}(q_{0})$ as a function of $q_{0}$ we can derive that there exists $r_{0}>0$ such that
\begin{itemize}
\item [$(h_{0})$]
the following inequality holds
\begin{equation*}
\dfrac{f(x,s)}{s}\leq q_{0}(x)+\varepsilon, \quad \text{for a.e. } x\in \mathopen{[}0,L\mathclose{]}, \; \forall \, 0<s\leq r_{0},
\end{equation*}
for every $\varepsilon>0$ such that $\mu_{1}(q_{0}+\varepsilon)>1$.
\end{itemize}

Now we can state the multiplicity result for positive solutions of the boundary value problem \eqref{BVP} (cf.~\cite[Theorem~4.1]{FeZa-15jde}).
We only give a sketch of the proof (for more details, we refer to \cite{FeZa-15jde}).

\begin{theorem}\label{Th-Multiplicity}
Let $f\colon\mathopen{[}0,L\mathclose{]}\times{\mathbb{R}}^{+}\to {\mathbb{R}}$ be an $L^{1}$-Carath\'{e}odory function
satisfying $(f^{*})$, $(f_{0}^{-})$, $(f_{0}^{+})$, $(H)$ and $(f_{\infty})$.
Let $r_{0}>0$ satisfy $(h_{0})$. Suppose that
\begin{enumerate}[label={$(\bigstar)$}]
\item \label{Star}
there exists $r\in\mathopen{]}0,r_{0}\mathclose{]}$ such that for every $\emptyset \neq \mathcal{I} \subseteq\{1,\ldots,m\}$
and every $L^{1}$-Carath\'{e}odory function
$h\colon \mathopen{[}0,L\mathclose{]}\times{\mathbb{R}}^{+}\to{\mathbb{R}}$ satisfying 
\begin{equation*}
\begin{aligned}
   &  h(x,s) \geq f(x,s), \quad \text{for a.e. } x\in \bigcup_{i\in {\mathcal{I}}} I_{i}, \; \forall \, s\geq 0,
\\ &  h(x,s) = f(x,s),    \quad \text{for a.e. } x\in \mathopen{[}0,L\mathclose{]} \setminus \bigcup_{i\in {\mathcal{I}}} I_{i}, \; \forall \, s\geq 0,
\end{aligned}
\end{equation*}
any non-negative solution $u(x)$ of
\begin{equation*}
u''+ h(x,u)=0
\end{equation*}
satisfies $\max_{x\in I_{i}} u(x) \neq r$ for every $i\in\{1,\ldots,m\}\setminus\mathcal{I}$.
\end{enumerate}
Then there exist at least $2^{m}-1$ positive solutions of the boundary value problem \eqref{BVP}.
\end{theorem}

\begin{proof}
We describe the main steps of the proof, adapting in the framework of our Sturm-Liouville boundary value problem the one given in \cite{FeZa-15jde}.

\medskip
\noindent\textit{Step 1. Equivalent fixed point problem. }
Using a standard procedure, we extend
$f(x,s)$ with the function $\tilde{f}\colon \mathopen{[}0,L\mathclose{]} \times \mathbb{R} \to \mathbb{R}$ defined as
\begin{equation*}
\tilde{f}(x,s)=
\begin{cases}
\, f(x,s), & \text{if } s\geq0; \\
\, 0,  & \text{if } s\leq0;
\end{cases}
\end{equation*}
and we consider the modified boundary value problem
\begin{equation}\label{BVPtilde}
\begin{cases}
\, u''+\tilde{f}(x,u)=0 \\
\, \alpha u(0) -\beta u'(0) = 0 \\
\, \gamma u(L) + \delta u'(L) = 0
\end{cases}
\end{equation}
(with $\alpha, \beta, \gamma, \delta \geq 0$ such that $\gamma\beta + \alpha\gamma +\alpha\delta > 0$).
As is well known, by a standard maximum principle (cf.~\cite[Lemma~2.1]{FeZa-15jde} and \cite{MaNjZa-95}), all the possible solutions of \eqref{BVPtilde} are non-negative and
hence solutions of \eqref{BVP}.

Next, we transform problem \eqref{BVPtilde} into an equivalent fixed point problem 
by means of the Green function associated to the equation $u''+u=0$ with the considered boundary conditions (cf.~\cite{ErHuWa-94,ErWa-94}), that is
\begin{equation*}
G(x,s) := \dfrac{1}{\gamma\beta + \alpha\gamma +\alpha\delta}
\begin{cases}
\, (\gamma+\delta-\gamma s)(\beta + \alpha x), & \text{if } \, 0 \leq x \leq s \leq 1; \\
\, (\gamma+\delta-\gamma x)(\beta + \alpha s), & \text{if } \, 0 \leq s \leq x \leq 1.
\end{cases}
\end{equation*}
Namely, we define $\Phi\colon\mathcal{C}(\mathopen{[}0,L\mathclose{]})\to\mathcal{C}(\mathopen{[}0,L\mathclose{]})$ as
\begin{equation}\label{operator}
(\Phi u)(x):= \int_{0}^{L} G(x,\xi)\tilde{f}(\xi,u(\xi)) ~\!d\xi, \quad x\in \mathopen{[}0,L\mathclose{]},
\end{equation}
and we notice that the operator $\Phi$ is completely continuous in $\mathcal{C}(\mathopen{[}0,L\mathclose{]})$ endowed with the $\sup$-norm $\|\cdot\|_{\infty}$.
Moreover, via a strong maximum principle, any nontrivial fixed point of $\Phi$ is a positive solution of \eqref{BVPtilde} and
hence a positive solution of \eqref{BVP}. Therefore we have reduced our problem to the search of nontrivial fixed points for the operator $\Phi$.

\medskip

Now we present how to reach the thesis using a topological approach based on the Leray-Schauder degree for locally compact operators defined on open possibly unbounded sets (cf.~\cite{Nu-85,Nu-93}).
We denote this degree with ``$\text{\rm deg}_{LS}$''.

\medskip
\noindent\textit{Step 2. Degree on small balls. }
Using conditions $(f^{*})$, $(f_{0}^{-})$ and $(f_{0}^{+})$, by a Sturm comparison argument, we observe that every non-negative solution $u(x)$ of the problem
\begin{equation}\label{BVP1}
\begin{cases}
\, u''+ \vartheta f(x,u)=0, \quad 0\leq\vartheta\leq1, \\
\, \alpha u(0) -\beta u'(0) = 0 \\
\, \gamma u(L) + \delta u'(L) = 0
\end{cases}
\end{equation}
satisfying $\max_{x\in \mathopen{[}0,L\mathclose{]}}u(x)\leq r_{0}$ is such that $u(x) = 0$, for all $x\in \mathopen{[}0,L\mathclose{]}$.
For the details, see \cite[Lemma~2.2]{FeZa-15jde}.

As an immediate consequence, by the homotopic invariance property of the degree, we have
\begin{equation*}
\text{\rm deg}_{LS}(Id-\Phi,B(0,r),0)=1, \quad \forall \, 0<r\leq r_{0},
\end{equation*}
wher $B(0,r)$ is the open ball in $\mathcal{C}(\mathopen{[}0,L\mathclose{]})$ centered at zero with radius $r$
(cf.~\cite[Lemma~2.3]{FeZa-15jde}).

\medskip
\noindent\textit{Step 3. Degree on large balls. }
From conditions $(f^{*})$, $(f_{0}^{-})$, $(H)$ and $(f_{\infty})$, we obtain the existence of $R^{*}>0$ such that
for each $L^{1}$-Carath\'{e}odory function $h\colon \mathopen{[}0,L\mathclose{]}\times{\mathbb{R}}^{+}\to{\mathbb{R}}$ with
\begin{equation*}
h(x,s) \geq f(x,s), \quad \text{for a.e. } x\in \bigcup_{i\in {\mathcal{I}}} I_{i}, \; \forall \, s\geq 0,
\end{equation*}
any non-negative solution $u(x)$ (defined in $\mathopen{[}0,L\mathclose{]}$) of the equation
\begin{equation*}
u''+ h(x,u)=0
\end{equation*}
satisfies $\max_{x\in I_{i}} u(x)\neq R^{*}$, for every $i=1,\ldots,m$.

We stress that the constant $R^{*}$ does not depend on the function $h(x,s)$. In particular the result holds for
$h(x,s) := f(x,s) + \alpha \mathbbm{1}_{A}(x)$, where $\alpha\geq0$ and $\mathbbm{1}_{A}$
denotes the indicator function of the set $A:=\bigcup_{i\in {\mathcal{I}}} I_{i}$.
For the details, see \cite[Lemma~2.4]{FeZa-15jde}.

As an immediate consequence, by the homotopic invariance property of the degree (cf.~\cite[Theorem~A.2]{FeZa-15jde}), we obtain
\begin{equation*}
\text{\rm deg}_{LS}(Id-\Phi,B(0,R),0)=0, \quad \forall \, R\geq R^{*},
\end{equation*}
(cf.~\cite[Lemma~2.5]{FeZa-15jde}).

\medskip
\noindent\textit{Step 4. Fixing the constants $r$ and $R$. }
We fix two positive constants $r,R\in\mathbb{R}$ such that 
\begin{equation*}
0<r\leq r_{0} < R^{*} \leq R 
\end{equation*}
and $r$ satisfying \ref{Star}. Then we consider the open unbounded sets $\Omega^{\mathcal{I}}=\Omega^{\mathcal{I}}_{r,R}$
and $\Lambda^{\mathcal{I}}=\Lambda^{\mathcal{I}}_{r,R}$, introduced in \eqref{eq-Omega} and in \eqref{eq-Lambda}, respectively.

\medskip
\noindent\textit{Step 5. Degree on $\Omega^{\mathcal{I}}$. }
For any subset of indices $\mathcal{I}\subseteq\{1,\ldots,m\}$ we compute $\text{\rm deg}_{LS}(Id-\Phi,\Omega^{\mathcal{I}},0)$.
First of all, let $\mathcal{I}=\emptyset$. By a convexity argument, we obtain that
\begin{equation}\label{deg-emptyset}
\text{\rm deg}_{LS}(Id-\Phi,\Omega^{\emptyset},0)=\text{\rm deg}_{LS}(Id-\Phi,B(0,R),0)=0.
\end{equation}
Secondly, let us consider a subset $\mathcal{I}\neq\emptyset$. Using a standard procedure, for any given $h(x,s)$ as in hypothesis \ref{Star}, we define a completely continuous operator
$\Psi_{h}\colon\mathcal{C}(\mathopen{[}0,L\mathclose{]})\to\mathcal{C}(\mathopen{[}0,L\mathclose{]})$ as
\begin{equation*}
(\Psi_{h} u)(x):= \int_{0}^{L} G(x,\xi)\tilde{h}(\xi,u(\xi))~\!d\xi, \quad x\in \mathopen{[}0,L\mathclose{]},
\end{equation*}
where
\begin{equation*}
\tilde{h}(x,s)=
\begin{cases}
\, h(x,s), & \text{if } s\geq0; \\
\, h(x,0),  & \text{if } s\leq0.
\end{cases}
\end{equation*}
Notice that $h(x,0) = 0$, for a.e.~$x\notin\bigcup_{i\in {\mathcal{I}}} I_{i}$, while $h(x,0) \geq 0$, for a.e.~$x\in\bigcup_{i\in {\mathcal{I}}} I_{i}$.
Hence, this extension of $h(x,s)$ allows us to apply a maximum principle.
Next, we observe that condition \ref{Star} ensures that $\Psi_{h}$ has no fixed points in $\partial\Omega^{\mathcal{I}}$ and so
the triplet $(Id-\Psi_{h},\Omega^{\mathcal{I}},0)$ is admissible.
Then, as in \textit{Step 3}, an homotopic argument gives
\begin{equation}\label{deg-1}
\text{\rm deg}_{LS}(Id-\Phi,\Omega^{\mathcal{I}},0)=0, \quad \forall \, \emptyset\neq\mathcal{I}\subseteq\{1,\ldots,m\},
\end{equation}
(cf.~\cite[Lemma~4.2]{FeZa-15jde}).

\medskip
\noindent\textit{Step 6. Degree on $\Lambda^{\mathcal{I}}$. }
For any subset of indices $\mathcal{I}\subseteq\{1,\ldots,m\}$ we now compute $\text{\rm deg}_{LS}(Id-\Phi,\Lambda^{\mathcal{I}},0)$.
Similarly as in \textit{Step 5}, from \ref{Star} we have that the triplet $(Id-\Phi,\Lambda^{\mathcal{I}},0)$ is admissible for every 
$\mathcal{I}\subseteq\{1,\ldots,m\}$.
Using an inductive argument (cf.~\cite[Lemma~4.1]{FeZa-15jde}), from \eqref{deg-emptyset} and \eqref{deg-1}, we obtain that 
\begin{equation*}
\text{\rm deg}_{LS}(Id-\Phi,\Lambda^{\mathcal{I}},0)=(-1)^{\#\mathcal{I}}, \quad \forall \, \mathcal{I}\subseteq\{1,\ldots,m\}.
\end{equation*}

\medskip
\noindent\textit{Step 7. Conclusion. }
Preliminarily, we underline that $0\notin\Lambda^{\mathcal{I}}$ for all $\emptyset\neq\mathcal{I}\subseteq\{1,\ldots,m\}$
and the sets $\Lambda^{\mathcal{I}}$ are pairwise disjoint. Since the number of nonempty subsets of a set with $m$ elements is $2^{m}-1$, there are 
$2^{m}-1$ sets $\Lambda^{\mathcal{I}}$ not containing the null function.
From \textit{Step 6}, in particular we have that
\begin{equation*}
\text{\rm deg}_{LS}(Id-\Phi,\Lambda^{\mathcal{I}},0)\neq0, \quad \forall \, \mathcal{I}\subseteq\{1,\ldots,m\}.
\end{equation*}
Therefore, by the existence property of the Leray-Schauder degree, there exist at least $2^{m}-1$ nontrivial fixed points of $\Phi$.
Finally, as already remarked, via a standard maximum principle argument, we obtain that these nontrivial fixed points are positive solutions
of \eqref{BVP}. The theorem follows.
\end{proof}

\section{Main multiplicity theorem}\label{section-4}

Recalling the setting and the notation introduced in Section~\ref{section-2}, 
now we state and prove the following main result.

\begin{theorem}\label{main-theorem}
Let $m\geq 1$ be an integer. Let $a_{i}\colon \mathopen{[}0,L\mathclose{]} \to \mathbb{R}^{+}$, for $i=1,\ldots,m$, 
and $b_{j} \colon \mathopen{[}0,L\mathclose{]} \to \mathbb{R}^{+}$, for $j=0,\ldots,m+1$, be Lebesgue integrable functions satisfying $(h_{1})$.
Let $g_{i}\colon \mathbb{R}^{+} \to \mathbb{R}^{+}$, for $i=1,\ldots,m$, 
and $k_{j}\colon \mathbb{R}^{+} \to \mathbb{R}^{+}$, $j=0,\ldots,m+1$, be continuous functions satisfying $(h_{2})$ and $(h_{3})$. Let $\alpha_{i}>0$, for all $i=1,\ldots,m$.
Moreover, suppose that 
\begin{equation}\label{hp_A}
\alpha_{i}g_{0}^{i} < \lambda_{0}, \quad \text{for all } \; i=1,\ldots,m,
\end{equation}
and
\begin{equation}\label{hp_B}
\alpha_{i}g_{\infty}^{i} > \lambda_{1}^{i}, \quad \text{for all } \; i=1,\ldots,m.
\end{equation}
Then there exists $\beta^{*}>0$ such that, if
\begin{equation*}
\beta_{j}>\beta^{*}, \quad \text{for all } \; j=0,\ldots,m+1,
\end{equation*}
the boundary value problem \eqref{BVP} with $f(x,s)$ defined in \eqref{def_f}
has at least $2^{m}-1$ positive solutions.
\end{theorem}

\begin{proof}
In order to prove the theorem, we are going to enter the setting of Theorem~\ref{Th-Multiplicity} and to check that all its hypotheses are satisfied for $\beta_{j}>0$ sufficiently large.

First of all, we observe that the map $f(x,s)$ defined as in \eqref{def_f} is an $L^{1}$-Carath\'{e}odory function and,
moreover, satisfies $(f^{*})$, due to condition $(h_{2})$.
From
\begin{equation*}
\liminf_{s\to0^{+}}\dfrac{f(x,s)}{s} 
 \geq \sum_{j=0}^{m+1} \beta_{j} b_{j}(x)\liminf_{s\to0^{+}} -\dfrac{k_{j}(s)}{s}
 = - \sum_{j=0}^{m+1} \beta_{j} b_{j}(x) k_{0}^{j},
\end{equation*}
for a.e.~$x\in\mathopen{[}0,L\mathclose{]}$,
and from the last assumption in $(h_{3})$, we deduce that $(f_{0}^{-})$ holds with $q_{-}\in L^{1}(\mathopen{[}0,L\mathclose{]},\mathbb{R}^{+})$
defined as
\begin{equation*}
q_{-}(x) := \sum_{j=0}^{m+1} \beta_{j} b_{j}(x) k_{0}^{j}, \quad x\in\mathopen{[}0,L\mathclose{]}.
\end{equation*}
For $i=1,\ldots,m$, by hypothesis \eqref{hp_A} let us fix $g_{*}^{i}>0$ such that $g_{0}^{i}<g_{*}^{i}<\lambda_{0}/\alpha_{i}$. Next, we define
\begin{equation*}
q_{0}(x) := \sum_{i=1}^{m} \alpha_{i} a_{i}(x)g_{*}^{i}, \quad x\in\mathopen{[}0,L\mathclose{]}. 
\end{equation*}
We observe that $q_{0}\in L^{1}(\mathopen{[}0,L\mathclose{]},\mathbb{R}^{+})$, $q_{0}\not\equiv0$ and
\begin{equation*}
\mu_{1}(q_{0}) \geq \mu_1 \biggl{(} \Bigl{(}\max_{i=1,\ldots,m} \alpha_{i}g_{*}^{i}\Bigr{)} \sum_{i=1}^{m}a_{i}(x)\biggr{)} 
= \dfrac{\lambda_{0}}{\displaystyle \max_{i=1,\ldots,m} \alpha_{i}g_{*}^{i}}>1.
\end{equation*}
Then condition $(f_{0}^{+})$ is valid.
Concerning the sign of $a(x)$, we observe that hypothesis $(H)$ directly follows from condition $(h_{1})$.
Furthermore, defining 
\begin{equation*}
q_{\infty}^{i}(x) := \alpha_{i} a_{i}(x) g_{\infty}^{i}, \quad x\in I_{i}, \quad \text{for } \, i=1,\ldots,m,
\end{equation*}
and observing that $q_{\infty}^{i}\in L^{1}(I_{i},\mathbb{R}^{+})$, $q_{\infty}^{i}\not\equiv0$ and
\begin{equation*}
\mu_{1}^{I_{i}}(q_{\infty}^{i}) = \dfrac{\lambda_{1}^{i}}{\alpha_{i}g_{\infty}^{i}} <1, \quad i=1,\ldots,m,
\end{equation*}
(by conditions $(h_{3})$ and \eqref{hp_B}), we obtain that $(f_{\infty})$ holds.

As a second step, we prove that hypothesis \ref{Star} is valid.
By condition \eqref{hp_A}, for all $i=1,\ldots,m$, we can choose $\rho_{i}>0$ such that $g_{0}^{i}<\lambda_{0}-\rho_{i}$.
As observed in Section~\ref{section-3}, by hypothesis $(f_{0}^{+})$ we can take $r_{0}>0$ satisfying $(h_{0})$ (as in Theorem~\ref{Th-Multiplicity}).
Next, we fix $0<r \leq r_{0}$ such that
\begin{equation}\label{eq-delta}
\alpha_{i}\dfrac{g_{i}(s)}{s} < \lambda_{0}-\rho_{i}, \quad \forall \, 0<s\leq r,
\end{equation}
(for  $i=1,\ldots,m$).
We claim that \ref{Star} holds for $r$ satisfying \eqref{eq-delta} and taking the parameters $\beta_{j}$ sufficiently large.
Let us consider an arbitrary set of indices $\emptyset\neq\mathcal{I}\subseteq\{1,\ldots,m\}$ and an arbitrary $L^{1}$-Carath\'{e}odory function $h(x,s)$ as in \ref{Star}.
Suppose by contradiction that there exists a non-negative solution $u(x)$ of $u''+ h(x,u)=0$ such that
\begin{equation*}
\max_{x\in I_{\ell}} u(x) = r, \quad \text{ for some index } \, \ell\in\{1,\ldots,m\}\setminus\mathcal{I}.
\end{equation*}

If $\mathcal{I}=\{1,\ldots,m\}$, there is nothing to prove. Then fix $\emptyset\neq\mathcal{I}\subsetneq\{1,\ldots,m\}$.
By the concavity of $u(x)$ in $I_{\ell}$, we have
\begin{equation}\label{eq-conc}
u(x) \geq \dfrac{r}{\tau_{\ell}-\sigma_{\ell}} \min\{x-\sigma_{\ell},\tau_{\ell}-x\},
\quad \forall \, x\in I_{\ell}=\mathopen{[}\sigma_{\ell},\tau_{\ell}\mathclose{]},
\end{equation}
(cf.~\cite[p.~420]{GaHaZa-03cpaa} for a similar estimate)

In order to prove that our assumption is contradictory, we split our argument into three steps.

\medskip
\noindent
\textit{Step 1. A priori bounds for $|u'(x)|$ on $I_{\ell}$.} 
First we notice that $h(x,u(x)) = \alpha_{\ell} a_{\ell}(x) g_{\ell}(u(x))$, for a.e. $x\in I_{\ell}$. Hence
\begin{equation*}
|u''(x)| \leq \eta_{r,\ell}(x):= \alpha_{\ell} a_{\ell}(x)\max_{0\leq s\leq r}g_{\ell}(s), \quad \text{for a.e. } x\in I_{\ell},
\end{equation*}
and, therefore
\begin{equation*}
|u'(y_{1}) - u'(y_{2})| \leq \|\eta_{r,\ell}\|_{L^{1}(I_{\ell})}, \quad \forall \, y_{1}, y_{2} \in I_{\ell}.
\end{equation*}
Since $0\leq u(x) \leq r$ for all $x\in I_{\ell}$, there exists a point $\hat{x}_{\ell} \in I_{\ell}$ such that
$|u'(\hat{x}_{\ell})| \leq r/(\tau_{\ell} - \sigma_{\ell})$.
Hence for all $x\in I_{\ell}$
\begin{equation}\label{eq-MK}
|u'(x)| \leq |u'(\hat{x}_{\ell})| + |u'(x) - u'(\hat{x}_{\ell})| \leq \dfrac{r}{\tau_{\ell} - \sigma_{\ell}} + \|\eta_{r,\ell}\|_{L^{1}(I_{\ell})} =: M_{\ell}.
\end{equation}

\medskip
\noindent
\textit{Step 2. Lower bounds for $u(x)$ on the boundary of $I_{\ell}$.}
Let $\varphi_{\ell}(x)$ be the positive eigenfunction of the eigenvalue problem on $I_{\ell}$
\begin{equation*}
\varphi''+\lambda_{1}^{\ell} a_{\ell}(x)\varphi=0, \quad \varphi|_{\partial I_{\ell}}=0,
\end{equation*}
with $\|\varphi_{\ell}\|_{\infty}=1$, where $\lambda_{1}^{\ell}>0$ is the first eigenvalue.
Then $\varphi_{\ell}(x)\geq0$, for all $x\in I_{\ell}$, $\varphi_{\ell}(x)>0$, for all $x\in\mathopen{]}\sigma_{\ell},\tau_{\ell}\mathclose{[}$,
and $\varphi_{\ell}'(\sigma_{\ell})>0>\varphi_{\ell}'(\tau_{\ell})$ (hence $\|\varphi_{\ell}'\|_{\infty}>0$).

By \eqref{eq-delta} and the fact that $\lambda_{0} \leq \lambda_{1}^{\ell}$, we know that
\begin{equation*}
\alpha_{\ell}g_{\ell}(s)<(\lambda_{1}^{\ell}-\rho_{\ell}) s, \quad \forall \, 0<s\leq r.
\end{equation*}
Then, using \eqref{eq-conc}, we have
\begin{equation*}
\begin{aligned}
   & \|\varphi_{\ell}'\|_{\infty}(u(\sigma_{\ell})+u(\tau_{\ell})) \geq
\\ & \geq u(\sigma_{\ell})\varphi_{\ell}'(\sigma_{\ell})+u(\tau_{\ell})|\varphi_{\ell}'(\tau_{\ell})|
= u(\sigma_{\ell})\varphi_{\ell}'(\sigma_{\ell})-u(\tau_{\ell})\varphi_{\ell}'(\tau_{\ell})
\\ & = \Bigl{[}u'(x)\varphi_{\ell}(x)-u(x)\varphi_{\ell}'(x)\Bigr{]}_{x=\sigma_{\ell}}^{x=\tau_{\ell}}
\\ & = \int_{\sigma_{\ell}}^{\tau_{\ell}}\dfrac{d}{dx}\Bigl{[}u'(x)\varphi_{\ell}(x)-u(x)\varphi_{\ell}'(x)\Bigr{]} ~\!dx
\\ & = \int_{I_{\ell}}\Bigl{[}u''(x)\varphi_{\ell}(x)-u(x)\varphi_{\ell}''(x)\Bigr{]} ~\!dx
\\ & = \int_{I_{\ell}}\Bigl{[}-h(x,u(x))\varphi_{\ell}(x)+u(x)\lambda_{1}^{\ell}a_{\ell}(x)\varphi_{\ell}(x)\Bigr{]} ~\!dx
\\ & = \int_{I_{\ell}}\Bigl{[}\lambda_{1}^{\ell}u(x)-\alpha_{\ell}g_{\ell}(u(x))\Bigr{]}a_{\ell}(x)\varphi_{\ell}(x) ~\!dx
\\ & > \int_{I_{\ell}}\rho_{\ell} \, \biggl{(}\dfrac{r}{\tau_{\ell}-\sigma_{\ell}} \min\{x-\sigma_{\ell},\tau_{\ell}-x\}\biggr{)}a_{\ell}(x)\varphi_{\ell}(x) ~\!dx
\\ & = r\, \biggl{[}\dfrac{\rho_{\ell}}{\tau_{\ell}-\sigma_{\ell}}\int_{I_{\ell}}\min\{x-\sigma_{\ell},\tau_{\ell}-x\} a_{\ell}(x)\varphi_{\ell}(x) ~\!dx\biggr{]}.
\end{aligned}
\end{equation*}
Hence, from the above inequality, we conclude that there exists a constant $c_{\ell}>0$, depending on $\rho_{\ell}$,
$I_{\ell}$ and $a_{\ell}(x)$, but independent on $u(x)$ and $r$, such that
\begin{equation*}
u(\sigma_{\ell}) + u(\tau_{\ell}) \geq {c}_{\ell}r > 0.
\end{equation*}
As a consequence of the above inequality, we have that at least one of the two inequalities
\begin{equation}\label{eq-ineq}
0 < \dfrac{c_{\ell}r}{2} \leq u(\tau_{\ell}) \leq r, \qquad  0 < \dfrac{c_{\ell}r}{2} \leq u(\sigma_{\ell}) \leq r,
\end{equation}
holds.

The two inequalities in \eqref{eq-ineq} reduce to a single one, if $\sigma_{1}=0$ and $\beta>0$, or if $\tau_{m}=L$ and $\delta>0$.
Indeed, if $\sigma_{1}=0$ and $\beta>0$, we have 
\begin{equation*}
\begin{aligned}
   & \Bigl{[}u'(x)\varphi_{1}(x)-u(x)\varphi_{1}'(x)\Bigr{]}_{x=0}^{x=\tau_{1}}
\\ & = u'(\sigma_{1})\varphi_{1}(\sigma_{1})-u(\sigma_{1})\varphi_{1}'(\sigma_{1})-u'(\tau_{1})\varphi_{1}(\tau_{1})+u(\tau_{1})\varphi_{1}'(\tau_{1})
\\ & = u'(\sigma_{1})\dfrac{\alpha}{\beta}\varphi_{1}'(\sigma_{1})-u(\sigma_{1})\varphi_{1}'(\sigma_{1})+u(\tau_{1})\varphi_{1}'(\tau_{1})
\\ & = u(\tau_{1})\varphi_{1}'(\tau_{1}) \leq \|\varphi_{1}'\|_{\infty}u(\tau_{1}).
\end{aligned}
\end{equation*}
Analogously, if $\tau_{m}=L$ and $\delta>0$, we have
\begin{equation*}
\Bigl{[}u'(x)\varphi_{1}(x)-u(x)\varphi_{m}'(x)\Bigr{]}_{x=\sigma_{m}}^{x=L} \leq \|\varphi_{m}'\|_{\infty}u(\sigma_{m}).
\end{equation*}
Finally, as a consequence of the previous inequalities, we obtain that
\begin{equation}\label{eq-ineq2}
0 < \dfrac{c_{\ell}r}{2} \leq u(\tau_{1}) \leq r \quad \text{ or } \quad  0 < \dfrac{c_{\ell}r}{2} \leq u(\sigma_{m}) \leq r
\end{equation}
holds, respectively.

\medskip
\noindent
\textit{Step 3. Contradiction on an adjacent interval for $\beta_{\ell}$ large.}
As a first case, we suppose that the first inequality in \eqref{eq-ineq} is true.
If $\tau_{\ell} = L$, then $\delta>0$ in the boundary conditions (otherwise $u(\tau_{\ell})=0$, a contradiction) and we deal with the second inequality in \eqref{eq-ineq2} (see the discussion of the second case below).
Consequently, whenever $\tau_{\ell} < L$, we can focus our attention on the right-adjacent interval $\mathopen{[}\tau_{\ell},\sigma_{\ell+1}\mathclose{]}$, where $f(x,u(x))=-\beta_{\ell}b_{\ell}(x)k_{\ell}(u(x))\leq 0$.
Recall also that, by the convention adopted in defining the intervals $I_{i}$ and $J_{j}$, 
we have that $b_{\ell}(x)$ is not identically zero on all right neighborhoods of $\tau_{\ell}$.

We observe that there exists $R>r$ such that $\max_{x\in \mathopen{[}0,L\mathclose{]}} u(x) < R$.
This is a consequence of $(f^{*})$, $(f_{0}^{-})$, $(H)$ and $(f_{\infty})$, as described in \textit{Step 3} of the proof of Theorem~\ref{Th-Multiplicity}.

Since $k_{\ell}(s) > 0$ for all $s > 0$, we can introduce the positive constant
\begin{equation*}
\nu_{\ell}:=\min_{\frac{{c}_{\ell}r}{4}\leq s\leq R}k_{\ell}(s)>0
\end{equation*}
and define
\begin{equation*}
\delta^{+}_{\ell}:= \min\biggl{\{}\sigma_{\ell+1}-\tau_{\ell},\dfrac{{c}_{\ell}r}{4M_{\ell}}\biggr{\}}>0,
\end{equation*}
where $M_{\ell}>0$ is the bound for $|u'(x)|$ obtained in \eqref{eq-MK} of \textit{Step 1}.
Then, by the convexity of $u(x)$ on $J_{\ell}$, we have that $u(x)$
is bounded from below by the tangent line at $(\tau_{\ell},u(\tau_{\ell}))$, with slope $u'(\tau_{\ell}) \geq -M_{\ell}$. Therefore,
\begin{equation*}
\dfrac{{c}_{\ell}r}{4}\leq u(x)\leq R, \quad \forall \, x\in\mathopen{[}\tau_{\ell},\tau_{\ell}+\delta^{+}_{\ell} \mathclose{]}.
\end{equation*}

We are going to prove that $\max_{x\in J_{\ell}}u(x)>R$ for $\beta_{\ell}>0$ sufficiently large 
(which is a contradiction with respect to the upper bound $R>0$ for $u(x)$).

Consider the interval $\mathopen{[}\tau_{\ell},\tau_{\ell}+\delta^{+}_{\ell}\mathclose{]}\subseteq J_{\ell}$.
For all $x\in \mathopen{[}\tau_{\ell},\tau_{\ell}+\delta^{+}_{\ell}\mathclose{]}$ we have
\begin{equation*}
u'(x) = u'(\tau_{\ell})+\int_{\tau_{\ell}}^{x}\beta_{\ell}b_{\ell}(\xi)k_{\ell}(u(\xi))~\!d\xi
      \geq -M_{\ell}+\beta_{\ell}\nu_{\ell}\int_{\tau_{\ell}}^{x} b_{\ell}(\xi)~\!d\xi,
\end{equation*}
then
\begin{equation*}
u(x) = u(\tau_{\ell})+\int_{\tau_{\ell}}^{x}u'(\xi)~\!d\xi
     \geq \dfrac{{c}_{\ell}r}{2}-M_{\ell}(x-\tau_{\ell})+\beta_{\ell}\nu_{\ell}\int_{\tau_{\ell}}^{x}\biggr(\int_{\tau_{\ell}}^{s} b_{\ell}(\xi)~\!d\xi\biggr) ~\!ds.
\end{equation*}
Hence, for $x=\tau_{\ell}+\delta^{+}_{\ell} $,
\begin{equation*}
R \geq  u(\tau_{\ell}+\delta^{+}_{\ell} ) \geq
\dfrac{{c}_{\ell}r}{2}-M_{\ell}\delta^{+}_{\ell}
+ \beta_{\ell}\nu_{\ell}\int_{\tau_{\ell}}^{\tau_{\ell}+\delta^{+}_{\ell}}\biggl{(}\int_{\tau_{\ell}}^{s} b_{\ell}(\xi)~\!d\xi \biggr{)} ~\!ds.
\end{equation*}
This gives a contradiction if $\beta_{\ell}$ is sufficiently large, say
\begin{equation*}
\beta_{\ell} > \beta_{\ell}^{+} := \dfrac{R + M_{\ell} L}{\nu_{\ell} \int_{\tau_{\ell}}^{\tau_{\ell}+\delta^{+}_{\ell}}\int_{\tau_{\ell}}^{s} b_{\ell}(\xi)~\!d\xi ~\!ds},
\end{equation*}
recalling that $\int_{\tau_{\ell}}^{x} b_{\ell}(\xi) ~\!d\xi > 0$ for each $x\in \mathopen{]}\tau_{\ell},\sigma_{k+1}\mathclose{]}$.

A similar argument (with obvious modifications) applies if the second inequality in \eqref{eq-ineq} is true.
If $\sigma_{\ell} = 0$, then $\beta>0$ in the boundary conditions (otherwise $u(\sigma_{\ell})=0$, a contradiction) and we deal with the first inequality in \eqref{eq-ineq2} (see the discussion of the first case above).
Consequently, whenever $\sigma_{\ell} > 0$, we can focus our attention
on the left-adjacent interval $J_{\ell-1}$ where $f(x,u(x))=-\beta_{\ell-1}b_{\ell-1}(x)k_{\ell-1}(u(x))\leq 0$. 
Recall also that, by the convention adopted in defining the intervals $I_{i}$ and $J_{j}$, 
we have that $b_{\ell-1}(x)$ is not identically zero on all left neighborhoods of $\sigma_{\ell}$.

If we define
\begin{equation*}
\delta^{-}_{\ell}  := \min\biggl{\{}\sigma_{\ell}-\tau_{\ell-1},\dfrac{{c}_{\ell}r}{4M_{\ell}}\biggr{\}}>0,
\end{equation*}
we obtain a similar contradiction for
\begin{equation*}
\beta_{\ell} > \beta_{\ell}^{-} := \dfrac{R + M_{\ell} L}{\nu_{\ell} \int_{\sigma_{\ell}-\delta^{-}_{\ell}}^{\sigma_{\ell}}\int_{s}^{\sigma_{\ell}} b_{\ell-1}(\xi) ~\!d\xi ~\!ds}.
\end{equation*}

At the end, defining
\begin{equation*}
\beta^{*}:= \max_{k=1,\ldots,m} \beta^{\pm}_{\ell},
\end{equation*}
condition \ref{Star} holds taking $\beta_{j}>\beta^{*}$, for all $j=0,\ldots,m+1$.
Finally, we can apply Theorem~\ref{Th-Multiplicity} and the proof is completed.
\end{proof}

From the statement of Theorem~\ref{main-theorem}, one can easily notice that the parameters $\alpha_{i}>0$ are involved only in hypotheses \eqref{hp_A} and \eqref{hp_B},
therefore there is no real condition on those constants (since they can be considered as part of the functions $g_{i}$).
In a moment, the role of the parameters $\alpha_{i}$ will become more clear. Indeed, investigating more on conditions \eqref{hp_A} and \eqref{hp_B},
we can state the following corollaries (the obvious proofs are omitted). 

\begin{corollary}\label{cor-4.1}
Let $m\geq 1$ be an integer. Let $a_{i}\colon \mathopen{[}0,L\mathclose{]} \to \mathbb{R}^{+}$, for $i=1,\ldots,m$, 
and $b_{j} \colon \mathopen{[}0,L\mathclose{]} \to \mathbb{R}^{+}$, for $j=0,\ldots,m+1$, be Lebesgue integrable functions satisfying $(h_{1})$.
Let $g_{i}\colon \mathbb{R}^{+} \to \mathbb{R}^{+}$, for $i=1,\ldots,m$, 
and $k_{j}\colon \mathbb{R}^{+} \to \mathbb{R}^{+}$, $j=0,\ldots,m+1$, be continuous functions satisfying $(h_{2})$ and $(h_{3})$.
Moreover, suppose that 
\begin{equation*}
g_{0}^{i}=0, \quad \text{for all } \; i=1,\ldots,m.
\end{equation*}
Then there exists $\alpha^{*}>0$ such that if
\begin{equation*}
\alpha_{i} > \alpha^{*}, \quad \text{for all } \; i=1,\ldots,m,
\end{equation*}
there exists $\beta^{*}=\beta^{*}(\alpha_{1},\ldots,\alpha_{m})>0$ so that, if
\begin{equation*}
\beta_{j}>\beta^{*}, \quad \text{for all } \; j=0,\ldots,m+1,
\end{equation*}
then the boundary value problem \eqref{BVP} with $f(x,s)$ defined in \eqref{def_f}
has at least $2^{m}-1$ positive solutions.
\end{corollary}

\begin{corollary}\label{cor-4.2}
Let $m\geq 1$ be an integer. Let $a_{i}\colon \mathopen{[}0,L\mathclose{]} \to \mathbb{R}^{+}$, for $i=1,\ldots,m$, 
and $b_{j} \colon \mathopen{[}0,L\mathclose{]} \to \mathbb{R}^{+}$, for $j=0,\ldots,m+1$, be Lebesgue integrable functions satisfying $(h_{1})$.
Let $g_{i}\colon \mathbb{R}^{+} \to \mathbb{R}^{+}$, for $i=1,\ldots,m$, 
and $k_{j}\colon \mathbb{R}^{+} \to \mathbb{R}^{+}$, $j=0,\ldots,m+1$, be continuous functions satisfying $(h_{2})$ and $(h_{3})$.
Moreover, suppose that 
\begin{equation*}
g_{\infty}^{i} = +\infty, \quad \text{for all } \; i=1,\ldots,m.
\end{equation*}
Then there exists $\alpha_{*}>0$ such that if
\begin{equation*}
0 < \alpha_{i} < \alpha_{*}, \quad \text{for all } \; i=1,\ldots,m,
\end{equation*}
there exists $\beta^{*}=\beta^{*}(\alpha_{1},\ldots,\alpha_{m})>0$ so that, if
\begin{equation*}
\beta_{j}>\beta^{*}, \quad \text{for all } \; j=0,\ldots,m+1,
\end{equation*}
then the boundary value problem \eqref{BVP} with $f(x,s)$ defined in \eqref{def_f}
has at least $2^{m}-1$ positive solutions.
\end{corollary}

\section{Positive radial solutions to elliptic BVPs}\label{section-5}

As a standard consequence of Theorem~\ref{main-theorem}, we can give a multiplicity result for positive radially symmetric
solutions to boundary value problems associated with elliptic PDEs on an annular domain.

We briefly describe the setting, referring to the notation introduced in Section~\ref{section-2}.
Let $0 < R_{1} < R_{2}$ and consider the open annulus around the origin
\begin{equation*}
\Omega := \bigl{\{}x\in {\mathbb{R}}^{N} \colon R_{1} < \|x\| < R_{2}\bigr{\}},
\end{equation*}
where $\|\cdot\|$ is the Euclidean norm in $\mathbb{R}^{N}$ (for $N \geq 2$).
We define
\begin{equation*}
\mathcal{F}(x,s):= \sum_{i=1}^{m} \alpha_{i} \mathcal{A}_{i}(x)g_{i}(s) - \sum_{j=0}^{m+1} \beta_{j} \mathcal{B}_{j}(x)k_{j}(s), \quad x\in\overline{\Omega}, \; s\in\mathbb{R}^{+},
\end{equation*}
with $m\geq1$.
For $i=1,\ldots,m$ and $j=0,1,\ldots,m+1$, let $\alpha_{i}>0$, $\beta_{j}>0$, and moreover let $g_{i}\colon \mathbb{R}^{+} \to \mathbb{R}^{+}$ and 
$k_{j}\colon \mathbb{R}^{+} \to \mathbb{R}^{+}$ be continuous functions satisfying conditions $(h_{2})$ and $(h_{3})$.
Let $\mathcal{A}_{i}\colon \overline{\Omega} \to \mathbb{R}^{+}$, for $i=1,\ldots,m$, and $\mathcal{B}_{j} \colon \overline{\Omega} \to \mathbb{R}^{+}$, for $j=0,1,\ldots,m+1$.

We deal with the Dirichlet boundary value problem associated with an elliptic partial differential equation
\begin{equation}\label{eq-pde-rad}
\begin{cases}
\, -\Delta \,u = \mathcal{F}(x,u) & \text{in } \Omega \\
\, u = 0 & \text{on } \partial\Omega.
\end{cases}
\end{equation}
For simplicity, we look for classical solutions to \eqref{eq-pde-rad}, namely, $u \in \mathcal{C}^2(\overline{\Omega})$.
Accordingly, we assume that $\mathcal{A}_{i}(x)$ and $\mathcal{B}_{j}(x)$ are \textit{continuous} functions.
Moreover, we suppose that $\mathcal{A}_{i}(x)$ and $\mathcal{B}_{j}(x)$ are radially symmetric function,
i.e.~there exist continuous functions $A_{i},B_{j} \colon \mathopen{[}R_{1},R_{2}\mathclose{]} \to \mathbb{R}^{+}$ such that
\begin{equation}\label{AB}
\mathcal{A}_{i}(x) = A_{i}(\|x\|), \quad \mathcal{B}_{i}(x) = B_{i}(\|x\|), \quad \forall \, x \in \overline{\Omega}.
\end{equation}
In this way, we can transform the partial differential equation in \eqref{eq-pde-rad} into a second order ordinary differential equation
as the one in \eqref{BVP}, as follows.

Preliminarily, we introduce the function
\begin{equation*}
F(r,s):= \sum_{i=1}^{m} \alpha_{i} A_{i}(r)g_{i}(s) - \sum_{j=0}^{m+1} \beta_{j} B_{j}(r)k_{j}(s), \quad r\in\mathopen{[}R_{1},R_{2}\mathclose{]}, \; s\in\mathbb{R}^{+}.
\end{equation*}
A radially symmetric (classical) solutions to \eqref{eq-pde-rad} is a solution of the
form $u(x) = \mathcal{U}(\|x\|)$, where $\mathcal{U}(r)$ is a scalar function defined on $\mathopen{[}R_{1},R_{2}\mathclose{]}$.
Consequently, we can convert \eqref{eq-pde-rad} into
\begin{equation}\label{eq-rad}
\begin{cases}
\, \bigl{(}r^{N-1}\, \mathcal{U}'\bigr{)}' + r^{N-1} F(r,\mathcal{U}) = 0 \\
\, \mathcal{U}(R_{1}) = \mathcal{U}(R_{2}) = 0.
\end{cases}
\end{equation}
Via the change of variable
\begin{equation*}
t = h(r):= \int_{R_{1}}^{r} \xi^{1-N} ~\!d\xi
\end{equation*}
and the positions
\begin{equation*}
L:= \int_{R_{1}}^{R_{2}} \xi^{1-N} ~\!d\xi, \quad r(t):= h^{-1}(t), \quad v(t)={\mathcal{U}}(r(t)),
\end{equation*}
we can transform \eqref{eq-rad} into the Dirichlet problem
\begin{equation*}
\begin{cases}
\, v'' + f(t,v) = 0 \\
\, v(0) = v(L) = 0,
\end{cases}
\end{equation*}
where
\begin{equation*}
f(t,v):= r(t)^{2(N-1)}F(r(t),v), \quad t \in \mathopen{[}0,T\mathclose{]}, \; v\in\mathbb{R}^{+}.
\end{equation*}

In this setting, a straightforward consequence of Theorem~\ref{main-theorem} is the following result.
In the statement below, when we introduce condition $(h_{1}^{*})$ and the points $\sigma_{i}$ and $\tau_{i}$, 
we implicitly assume the convention adopted in defining the intervals $I_{i}$ and $J_{j}$ in Section~\ref{section-2}.

\begin{theorem}\label{th-rad}
Let $m\geq 1$ be an integer. Let $\mathcal{A}_{i}\colon \overline{\Omega} \to \mathbb{R}^{+}$, for $i=1,\ldots,m$, 
and $\mathcal{B}_{j}\colon \overline{\Omega} \to \mathbb{R}^{+}$, for $j=0,\ldots,m+1$, be Lebesgue integrable functions satisfying
the following condition:
\begin{itemize}
\item [$(h_{1}^{*})$]
\textit{there exist $2m + 2$ points (with $m \geq 1$)
\begin{equation*}
R_{1} = \tau_{0} \leq \sigma_{1} < \tau_{1} < \sigma_{2} < \ldots  < \tau_{m-1} < \sigma_{m} < \tau_{m} \leq \sigma_{m+1} = R_{2},
\end{equation*}
such that $A_{i}\not\equiv 0$ on $\mathopen{[}\sigma_{i},\tau_{i}\mathclose{]}$, for $i=1,\ldots,m$,
and $B_{i}\not\equiv 0$ on $\mathopen{[}\tau_{i},\sigma_{i+1}\mathclose{]}$, for $j=1,\ldots,m+1$,}
\end{itemize}
where $A_{i},B_{j} \colon \overline{\Omega} \to \mathbb{R}^{+}$ are defined as in \eqref{AB}.
Let $\alpha_{i}>0$, for all $i=1,\ldots,m$.
Let $g_{i}\colon \mathbb{R}^{+} \to \mathbb{R}^{+}$, for $i=1,\ldots,m$, 
and $k_{j}\colon \mathbb{R}^{+} \to \mathbb{R}^{+}$, $j=0,\ldots,m+1$, be continuous functions satisfying $(h_{2})$, $(h_{3})$, \eqref{hp_A} and \eqref{hp_B}.
Then there exists $\beta^{*}>0$ such that, if
\begin{equation*}
\beta_{j}>\beta^{*}, \quad \text{for all } \; j=0,\ldots,m+1,
\end{equation*}
the Dirichlet boundary value problem \eqref{eq-pde-rad} has at least $2^{m}-1$ positive radially symmetric
(classical) solutions.
\end{theorem}

Clearly, from Corollary~\ref{cor-4.1} and Corollary~\ref{cor-4.2} we also derive the following result.

\begin{corollary}\label{cor-5.1}
Let $m\geq 1$ be an integer. Let $\mathcal{A}_{i}\colon \overline{\Omega} \to \mathbb{R}^{+}$, for $i=1,\ldots,m$, 
and $\mathcal{B}_{j}\colon \overline{\Omega} \to \mathbb{R}^{+}$, for $j=0,\ldots,m+1$, be Lebesgue integrable functions satisfying $(h_{1}^{*})$.
Let $g_{i}\colon \mathbb{R}^{+} \to \mathbb{R}^{+}$, for $i=1,\ldots,m$, 
and $k_{j}\colon \mathbb{R}^{+} \to \mathbb{R}^{+}$, $j=0,\ldots,m+1$, be continuous functions satisfying $(h_{2})$ and $(h_{3})$.
\begin{itemize}
\item If
\begin{equation*}
g_{0}^{i}=0, \quad \text{for all } \; i=1,\ldots,m.
\end{equation*}
Then there exists $\alpha^{*}>0$ such that if
\begin{equation*}
\alpha_{i} > \alpha^{*}, \quad \text{for all } \; i=1,\ldots,m,
\end{equation*}
there exists $\beta^{*}=\beta^{*}(\alpha_{1},\ldots,\alpha_{m})>0$ so that, if
\begin{equation*}
\beta_{j}>\beta^{*}, \quad \text{for all } \; j=0,\ldots,m+1,
\end{equation*}
then the Dirichlet boundary value problem \eqref{eq-pde-rad} has at least $2^{m}-1$ positive radially symmetric
(classical) solutions.

\item If
\begin{equation*}
g_{\infty}^{i} = +\infty, \quad \text{for all } \; i=1,\ldots,m.
\end{equation*}
Then there exists $\alpha_{*}>0$ such that if
\begin{equation*}
0 < \alpha_{i} < \alpha_{*}, \quad \text{for all } \; i=1,\ldots,m,
\end{equation*}
there exists $\beta^{*}=\beta^{*}(\alpha_{1},\ldots,\alpha_{m})>0$ so that, if
\begin{equation*}
\beta_{j}>\beta^{*}, \quad \text{for all } \; j=0,\ldots,m+1,
\end{equation*}
then the Dirichlet boundary value problem \eqref{eq-pde-rad} has at least $2^{m}-1$ positive radially symmetric
(classical) solutions.
\end{itemize}
\end{corollary}

We conclude this discussion by observing that the multiplicity results given in Theorem~\ref{th-rad} and in its corollary are also valid considering different boundary conditions of the form
\begin{equation*}
u=0 \; \text{ on } \bigl{\{}x\in {\mathbb{R}}^{N} \colon \|x\| = R_{1}\bigr{\}} 
\quad \text{ and } \quad
\dfrac{\partial u}{\partial r}  =0 \; \text{ on } \bigl{\{}x\in {\mathbb{R}}^{N} \colon \|x\| = R_{2}\bigr{\}},
\end{equation*}
or
\begin{equation*}
\dfrac{\partial u}{\partial r} =0 \; \text{ on } \bigl{\{}x\in {\mathbb{R}}^{N} \colon \|x\| = R_{1}\bigr{\}} 
\quad \text{ and } \quad
u=0 \; \text{ on } \bigl{\{}x\in {\mathbb{R}}^{N} \colon \|x\| = R_{2}\bigr{\}},
\end{equation*}
where $r=\|x\|$ and $\partial u / \partial r$ denotes the differentiation in the radial direction
(compare also to \cite{LaWe-98}, where an existence result for positive solutions is given for this type of conditions).

\section{Final remarks}\label{section-6}

In this final section we present some consequences and discussions that naturally arise from our main result and from the topological approach developed in this paper, when compared to the existing literature.

\medskip

As the first point, in order to better explain our contribution to indefinite problems, we compare our main result to the one given in \cite{FeZa-15jde}.
In \cite{FeZa-15jde} the authors presented an application of Theorem~\ref{Th-Multiplicity} (i.e.~\cite[Theorem~4.1]{FeZa-15jde}) to an indefinite equation of the form
\begin{equation}\label{eq-jde}
u'' + a(x)g(u) = 0,
\end{equation}
where $a(x)\geq 0$ on $m$ pairwise disjoint intervals and $a(x)\leq 0$ on the complement in $\mathopen{[}0,L\mathclose{]}$.
According to our notation, setting $a_{i}:=a|_{I_{i}}$ and $b_{j}:=a|_{J_{j}}$,
one can easily see that \cite[Theorem~5.3]{FeZa-15jde} is an immediate consequence of Theorem~\ref{main-theorem}.
Furthermore, we observe that in the special case of \eqref{eq-jde} Theorem~\ref{main-theorem} generalizes \cite[Theorem~5.3]{FeZa-15jde}.
Indeed, in the present paper, the positive part and the negative part of the weight are associated with different nonlinearities, that is $g_{i}(s)$ and $k_{j}(s)$.
This fact allows us to impose growth conditions only on the nonlinearities that have actually a role in the proof. 
More precisely, we assume superlinear growth conditions at zero and at infinity on the nonlinearities $g_{i}(s)$ (that multiply the positive part of the weight),
while there are no growth conditions on the nonlinearities associated with the non-negative part.
Indeed, besides the standard sign condition $(h_{2})$, we assume only that $k^{j}_{0}<+\infty$ (in $(h_{3})$) in order to apply a standard maximum principle.
In Figure~\ref{fig-02} we show an example of equation which does not enter the setting of \cite[Theorem~5.3]{FeZa-15jde}, while it satisfies all the hypotheses of Theorem~\ref{main-theorem}.

\begin{figure}[h!]
\centering
\centering
\begin{tikzpicture}[x=40pt,y=25pt]
\draw (-0.4,0) -- (6,0);
\draw (0,-1.4) -- (0,1.5);
\draw (5.8,0) node[anchor=south] {$x$};
\draw [line width=1pt, color=red] (2,0) sin (2.5,-1) cos (3,0);
\draw [line width=1pt, color=red] (3,0) -- (4,0);
\draw [line width=1pt, color=red] (4,0) sin (4.5,1) cos (5,0);
\draw (0,0) node[anchor=north east] {$0$};
\draw (1,-0.1) node[anchor=north] {$1$};
\draw (2.1,-0.1) node[anchor=north east] {$2$};
\draw (2.9,-0.1) node[anchor=north west] {$3$};
\draw (4,-0.1) node[anchor=north] {$4$};
\draw (5,-0.1) node[anchor=north] {$5$};
\draw [line width=1pt, color=red] (0,1) -- (2,1);
\draw [dashed] (2,-0.1) -- (2,1);
\draw (1,-0.1) -- (1,0);
\draw (2,-0.1) -- (2,0);
\draw (3,-0.1) -- (3,0);
\draw (4,-0.1) -- (4,0);
\draw (5,-0.1) -- (5,0);
\draw (0,1) node[anchor=east] {$1$};
\draw (-0.03,1) -- (0,1);
\draw (0,-1) node[anchor=east] {$-1$};
\draw (-0.03,-1) -- (0,-1);
\end{tikzpicture}
\vspace*{10pt}
\\
\includegraphics[width=0.305\textwidth]{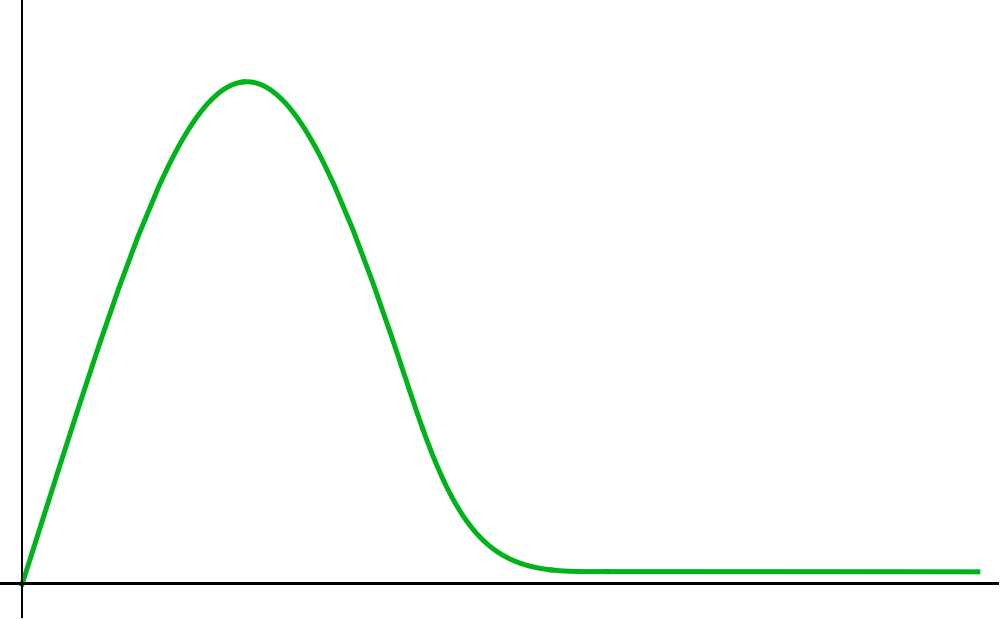} \quad
\includegraphics[width=0.305\textwidth]{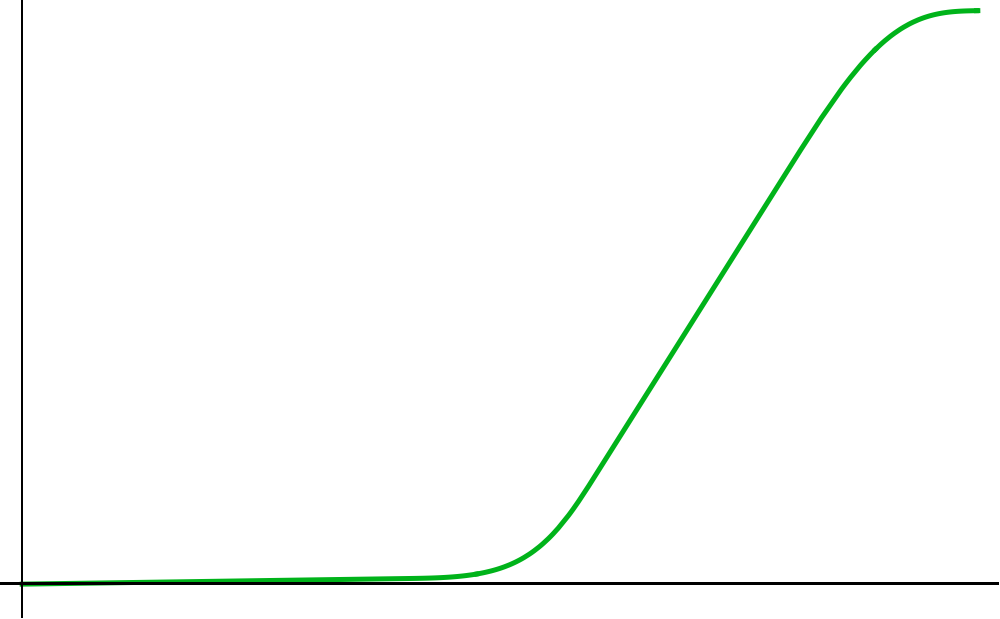} \quad
\includegraphics[width=0.305\textwidth]{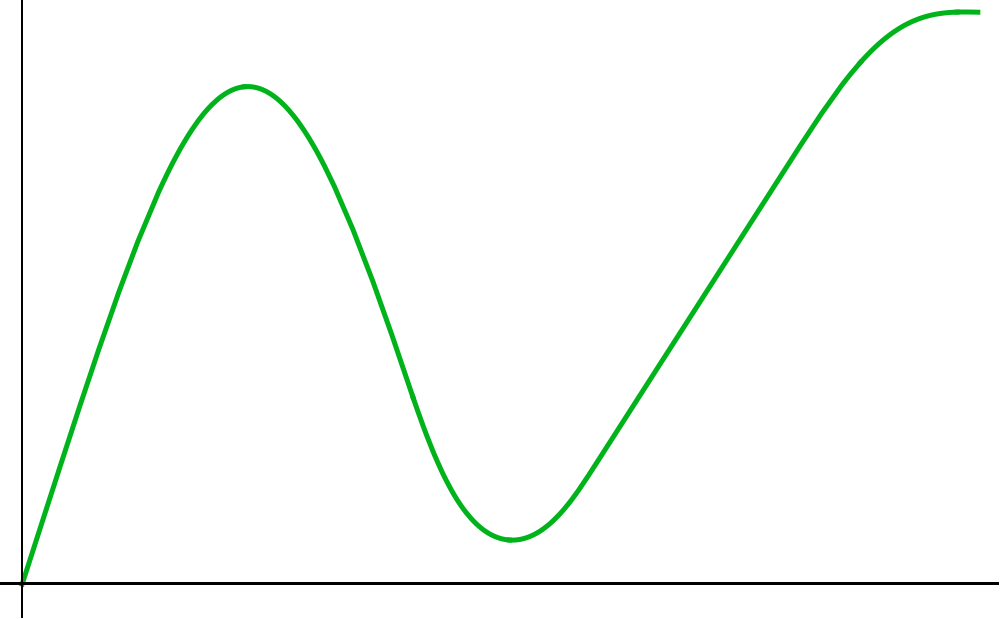}
\caption{\small{The figure shows an example of $3$ positive solutions to the equation $u''+\alpha_{1}a_{1}(x)g_{1}(u)-\beta_{1}b_{1}(x)k_{1}(u)+\alpha_{2}a_{2}(x)g_{2}(u)=0$ on $\mathopen{[}0,5\mathclose{]}$ with $u(0)=u'(5)=0$, whose graphs are located in the lower part of the figure.
For this simulation we have chosen $\alpha_{1}=10$, $\alpha_{2}=2$, $\beta_{1}=20$ and the weight functions as in the upper part of the figure, that is 
$a_{1}(x)=1$ in $\mathopen{[}0,2\mathclose{]}$, $-b_{1}(x)=-\sin (\pi x)$ in $\mathopen{[}2,3\mathclose{]}$, $a_{2}(x)=0$ in $\mathopen{[}3,4\mathclose{]}$,
$a_{2}(x)=-\sin (\pi x)$ in $\mathopen{[}4,5\mathclose{]}$.
Moreover, we have taken $g_{1}(s)=g_{2}(s)=s\arctan(s)$ and $k_{1}(s)=s/(1+s^{2})$ (for $s>0$).
Notice that $k_{1}(s)$ has not a superlinear behavior, since $\lim_{s\to 0^{+}}k_{1}(s)/s=1>0$ and $\lim_{s\to +\infty}k_{1}(s)/s=0$.
Then \cite[Theorem~5.3]{FeZa-15jde} does not apply, contrary to Theorem~\ref{main-theorem}.
}}
\label{fig-02}
\end{figure}
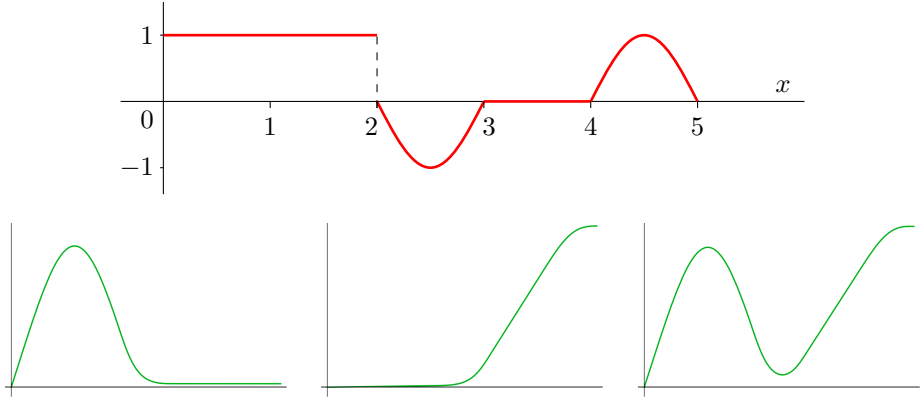

\medskip

One of the advantages in using an approach based on the topological degree is the fact that the degree
is stable with respect to small perturbation of the operator and hence our multiplicity result is valid also when we consider an equation of the form
\begin{equation*}
u'' + \varepsilon p(t,u,u') + f(x,u) = 0,
\end{equation*}
for $|\varepsilon|$ sufficiently small.

From this remark we immediately obtain that we can deal with the equation
\begin{equation*}
u'' + \lambda u + f(x,u) = 0
\end{equation*}
for $|\lambda|$ small enough and thus providing a contribution to \cite{AlTa-96} (compare to the discussion in the introduction).
Moreover, we can consider the Sturm-Liouville problem associated with
\begin{equation}\label{eq-cu'}
u'' + cu' + f(x,u) = 0,
\end{equation}
where $c\in\mathbb{R}$ is a constant, with $|c|$ small enough.
The above equation has no Hamiltonian structure. An interesting question is whether Theorem~\ref{main-theorem} is still valid for an arbitrary $c\in\mathbb{R}$.
With Dirichlet boundary conditions or mixed boundary conditions of the form $u'(0)=u(T)=0$ or $u(0)=u'(T)=0$, a standard change of variable allows
to reduce equation \eqref{eq-cu'} to an equation of the form considered in this paper; while for the general case of Sturm-Liouville boundary conditions one can adapt the approach developed in \cite{FeZa-15ade,FeZa-pp2015} (introduced for Neumann and periodic problems).
In \cite{FeZa-15ade,FeZa-pp2015} the authors used suitable monotonicity properties of the map $t\mapsto e^{ct}u'(t)$ that replace the convexity/concavity of the solutions of \eqref{BVP}.
However, in order to avoid unnecessary technicalities, we prefer to skip further investigations in this direction.

\medskip

Another question that naturally arise is whether we can consider other boundary condition, as the Neumann and periodic ones.
In the case of Neumann and periodic boundary conditions the linear differential operator $u\mapsto -u''$ has a nontrivial kernel made up of the constant functions. Then the operator is not invertible and we cannot proceed as explained in Section~\ref{section-2} defining an equivalent fixed point problem by means of the Green function.
A first possibility is that of perturbing the linear differential operator to a new one which is invertible and next one have to recover the original equation via a limiting process and some careful estimates on the solutions. A second possibility, described in \cite{FeZa-15ade,FeZa-pp2015}, is to apply the coincidence degree
theory developed by J.~Mawhin, which allows to study equations of the form
$Lu = Nu$, where $L$ is a linear operator with nontrivial kernel and $N$ is a
nonlinear one.

\medskip

A procedure analogous to the one described in this paper can be used to prove multiplicity results for solution to \eqref{BVP}, when roughly speaking $s\mapsto f(x,s)$ has a superlinear growth at zero and a \textit{sublinear} growth at infinity.
In this \textit{super-sublinear} case, following the theory developed in \cite{BoFeZa-17tams}, hypotheses \eqref{hp_A} and \eqref{hp_B} of Theorem~\ref{main-theorem} are replaced by
\begin{equation*}
g_{0}^{i} = g_{\infty}^{i} = 0, \quad \text{for all } \; i=1,\ldots,m.
\end{equation*}
It seems probable to prove the existence of $3^{m}-1$ positive solutions when $\alpha_{i}$ and $\beta_{j}$ are sufficiently large.

\section*{Acknowledgments}

This work benefited from fruitful discussions, helpful suggestions and encouragement of Professor Fabio Zanolin.

\bibliographystyle{elsart-num-sort}
\bibliography{Feltrin_biblio}

\bigskip
\begin{flushleft}

{\small{\it Preprint}}

{\small{\it July 2016}}

\end{flushleft}

\end{document}